\documentclass[11pt]{amsart}

\usepackage[sans]{dsfont}
\usepackage{a4wide}
\usepackage{amsfonts,amssymb,amsbsy,amsmath,amsthm,enumerate}
\usepackage{color}
\usepackage{epsfig}
\usepackage{graphicx}
\usepackage{mathrsfs}

\numberwithin{equation}{section}

\newtheorem{thm}{Theorem}[section]
\newtheorem{proposition}[thm]{Proposition}
\newtheorem{cor}[thm]{Corollary}
\newtheorem{lem}[thm]{Lemma}

\theoremstyle{definition}

\theoremstyle{remark}
\newtheorem*{xrem}{Remark}

\def\beq{\begin{equation}}
\def\eeq{\end{equation}}
\newcommand{\bea}{\begin{eqnarray}}
\newcommand{\eea}{\end{eqnarray}}
\newcommand{\beas}{\begin{eqnarray*}}
\newcommand{\eeas}{\end{eqnarray*}}

\def\<{\langle}
\def\>{\rangle}

\DeclareMathOperator{\CAP}{\it CAP}
\DeclareMathOperator{\curl}{curl}
\DeclareMathOperator{\dist}{dist}
\DeclareMathOperator{\dive}{div}
\DeclareMathOperator{\Dom}{Dom}
\DeclareMathOperator{\modu}{mod}
\DeclareMathOperator{\Tr}{Tr}
\DeclareMathOperator{\WAP}{\it WAP}

\DeclareMathOperator{\Ker}{Ker}

\DeclareMathOperator{\osc}{osc}

\DeclareMathOperator{\re}{Re}
\DeclareMathOperator{\im}{Im}
\DeclareMathOperator{\Trig}{Trig}
\DeclareMathOperator{\supp}{supp}

\newcommand{\one}{\mathds{1}}
\newcommand{\bel}{\begin{equation} \label}
\newcommand{\ee}{\end{equation}}

\newcommand{\eps}{{\epsilon}}
\newcommand{\veps}{{\varepsilon}}

\newcommand{\rd}{{\mathbb R}^{2}}
\newcommand{\bd}{{\mathbb B}^{2}}

\def\B{{\mathbb B}}
\def\C{{\mathbb C}}

\def\E{{\mathbb E}}

\def\N{{\mathbb N}}
\newcommand{\PPP}{{\mathbb P}}

\def\R{{\mathbb R}}
\newcommand{\SSS}{{\mathbb S}}
\def\T{{\mathbb T}}
\def\Z{{\mathbb Z}}

\def\CA{\mathcal {A}}
\def\CB{\mathcal {B}}

\def\CF{\mathcal {F}}

\def\CH{\mathcal {H}}
\def\CI{\mathcal {I}}

\def\CM{\mathcal {M}}

\def\CO{\mathcal {O}}

\def\CT{\mathcal {T}}
\def\CU{\mathcal {U}}
\def\CV{\mathcal {V}}

\def\SA{{\mathscr A}}
\def\SD{{\mathscr D}}
\def\SH{{\mathscr H}}
\def\SC{{\mathscr C}}

\begin{document}
\title[Almost periodic Pauli operators]
{Spectral properties of 2D Pauli operators with almost periodic electromagnetic fields}

\author[J.-F. Bony]{Jean-Fran\c{c}ois Bony}
\address{Institut de Math\'ematiques de Bordeaux, UMR 5251 du CNRS,
Universit\'e de Bordeaux, 351 cours de la Lib\'eration, 33405 Talence cedex, France}
\email{bony@math.u-bordeaux.fr}
\author[N.~Espinoza]{Nicol\'as Espinoza}
\address{Graduate School of Mathematical Sciences, University of Tokyo
3-8-1, Komaba, Meguro-ku, Tokyo 153-8914, Japan}
\email{nespino@ms.u-tokyo.ac.jp}
\author[G. Raikov]{Georgi Raikov}
\address{Departamento de Matem\'aticas, Facultad de Matem\'aticas, Pontificia Universidad
Cat\'olica de Chile, Vicu\~na Mackenna 4860, Santiago de Chile, Chile}
\email{graikov@mat.uc.cl}

\begin{abstract}
We consider a 2D Pauli operator with almost periodic field $b$ and electric potential $V$. First, we study the ergodic properties of $H$ and show, in particular, that its discrete spectrum is empty if there exists an almost periodic magnetic potential which generates the magnetic field $b - b_{0}$, $b_{0}$ being the mean value of $b$. Next, we assume that $V = 0$, and investigate the zero modes of $H$. As expected, if $b_{0} \neq 0$, then generically $\dim \Ker H = \infty$. If $b_{0} = 0$, then for each $m \in \N \cup \{ \infty \}$, we construct almost periodic $b$ such that $\dim \Ker H = m$. This construction depends strongly on results concerning the asymptotic behavior of Dirichlet series, also obtained in the present article.
\end{abstract}

\maketitle



\section{Introduction}

In the present article we study the spectral properties of the 2D Pauli operator $H$ with scalar magnetic field $b$ and electric potential $V$.

First, we assume that $b$ and $V$ are almost periodic and there exists an almost periodic magnetic potential $\widetilde{A}$ which generates $\widetilde{b} : = b - b_{0}$, $b_{0}$ being the mean value of $b$, and construct an ergodic family of operators $\{ \CH_{\omega} \}_{\omega \in \B^{2}}$ such that $\CH_{0} = H$. Here $\B^{2}$ is the Bohr compactification of $\R^{2}$, equipped with the normalized Haar measure $\PPP$. Using the general properties of ergodic families of operators, and the uniform continuity  of the resolvent $( \CH_{\omega} - z )^{- 1}$, $z \in \C \setminus \R$, with respect to  $\omega \in \B^{2}$, we show that for every $\omega \in \B^{2}$ the spectrum $\sigma ( \CH_{\omega} )$ of $\CH_{\omega}$ is the same and the discrete spectrum $\sigma_{\rm disc} ( \CH_{\omega} )$ is empty, while the absolutely continuous spectrum $\sigma_{\rm ac} ( \CH_{\omega} )$, the singular continuous spectrum $\sigma_{\rm sc} ( \CH_{\omega} )$, and the closure $\sigma_{\rm pp} ( \CH_{\omega} )$ of the set of the eigenvalues of $\CH_{\omega}$, are almost surely constant. Moreover, we prove that almost surely any fixed $E \in \R$ is not an eigenvalue of $\CH_{\omega}$ of finite multiplicity. Next, we assume only that $b$ and $V$ are almost periodic without supposing the existence of an almost periodic $\widetilde{A}$ which generates $\widetilde{b}$, and extend to this case the above results which now all hold almost surely.

Further, we investigate the kernel of the operator $H$ with $V = 0$. We concentrate on the problem of determining $\dim \Ker H$ for a given magnetic field $b$. If $\Ker H$ is not trivial, we also address the issue of whether the zero is an isolated point of $\sigma ( H )$. To start with, we recall the classical results concerning rapidly decaying or periodic $b$, and then we pass to almost periodic fields. First, we consider a class of such fields which was studied already in \cite{r5}, and, in a certain sense, is close to the class of periodic $b$. We recall that for this class we have $\dim \Ker H = \infty$ if $b_{0} \neq 0$,  the zero is an isolated point of $\sigma ( H )$, and an effective bound of the size of the spectral gap adjoining the origin is available, while if $b_{0} = 0$, then $\Ker H = \{ 0 \}$. Further, we consider almost periodic magnetic fields which are distant from the periodic ones. In this case again $\dim \Ker H = \infty$ if $b_{0} \neq 0$. However, if $b_{0} = 0$ the situation changes drastically in comparison with the previous class. Namely, for each $m \in \N \cup \{ \infty \}$ we construct explicitly magnetic fields for which $\dim \Ker H = m$. If $m \in \N$, then, due to the ergodic properties of $H$, the zero is not an isolated point of $\sigma ( H )$. Our construction strongly relies on some new results concerning the asymptotic behavior of certain Dirichlet series containing on a large parameter, which are also obtained in the present article; these results could be of independent interest, say, in the analytic number theory.

The article is organized as follows. In Section \ref{s2} we introduce the 2D Pauli operator $H$ and describe some of its general properties such as its supersymmetric form in the case $V = 0$, as well as its gauge invariance. In Section  \ref{s3}, we discuss the ergodic properties of $H$. Finally, in Section \ref{s4} we assume $V = 0$, and investigate the zero modes of $H$. Finally, the Appendix contains the proofs of our results on the asymptotics of Dirichlet series.

\section{Two-dimensional Pauli operators: general setting} \label{s2}

Let $b : \R^{2} \to \R$ be a bounded continuous function which has the physical interpretation of a scalar magnetic field, and let $A = ( A_{1} , A_{2} ) \in C^{1} ( \R^{2} ; \R^{2})$ be a vector field such that
\begin{equation*}
b = \curl A : = \frac{\partial A_{2}}{\partial x_{1}} - \frac{\partial A_{1}}{\partial x_{2}} .
\end{equation*}
Then, $A$ is interpreted as a magnetic potential which generates the magnetic field $b$. Let $M_{2}$ be the set of Hermitian $2 \times 2$ matrices, and let $V: \R^{2} \to M_{2}$ be a bounded continuous function, interpreted as a matrix-valued electric potential. Then the 2D Pauli operator $H$ with magnetic potential $A$ and electric potential $V$, acting in the Hilbert space $L^{2} ( \R^{2} ; \C^{2} )$,  can be defined as
\begin{equation*}
H = H ( A , V ) = \left(
\begin{array}{cc}
H^{-} & 0 \\
0 & H^{+}
\end{array}
\right) + V ,
\end{equation*}
where $H^{\pm} : = h \pm b$ and $h = h( A ) : = ( - i \nabla - A )^{2}$ is the 2D Schr\"odinger operator with magnetic potential $A$. Let us recall that for $A \in L^{2}_{\rm loc} ( \R^{2} ; \R^{2} )$, the operator $h$ can be defined as the self-adjoint operator generated by the closure of the quadratic form
\begin{equation*}
\int_{\R^{2}} \vert i \nabla u + A u \vert^{2} d x , \quad u \in C_{0}^{\infty} ( \R^{2} ) .
\end{equation*}
If
\begin{equation} \label{j1}
A \in L^{4}_{\rm loc} ( \R^{2} ; \R^{2} ) , \qquad \dive A \in L^{2}_{\rm loc} ( \R^{2} ) ,
\end{equation}
then $h$ is essentially self-adjoint on $C^{\infty}_{0} ( \R^{2} )$ (see \cite{ls}). Note that $A = ( A_{1} , A_{2} ) \in C^{1} ( \R^{2} ; \R^{2} )$ implies \eqref{j1}. Thus, the block operators $H^{\pm}$ with common domain $\Dom H^{\pm} = \Dom h$ are self-adjoint in $L^{2} ( \R^{2} )$ and the matrix operator $H ( A , V )$ is self-adjoint in $L^{2} ( \R^{2} ; \C^{2} )$.
Let us introduce the magnetic creation operator
\begin{equation*}
a^{*} = a ( A )^{*} : = - 2 i \frac{\partial}{\partial z} - A_{1} + i A_{2}, \quad z = x_{1} + i x_{2} ,
\end{equation*}
and the magnetic annihilation operator
\begin{equation} \label{au3}
a = a ( A ) : = - 2 i \frac{\partial}{\partial \overline{z}} - A_{1} - i A_{2} , \quad \overline{z} = x_{1} - i x_{2} .
\end{equation}
The operators $a$ and $a^{*}$ with common domain $\Dom h^{1 / 2}$ are closed and mutually adjoint in $L^{2} ( \R^{2} )$. Then the operator $H ( A , 0 )$ can be written in the supersymmetric form
\begin{equation} \label{j2}
H ( A , 0 ) = \left(
\begin{array}{cc}
0 & a^{*}\\
a & 0
\end{array}
\right)^{2} = \left(
\begin{array}{cc}
a^{*}a & 0\\
0 & a a^{*}
\end{array}
\right) ,
\end{equation}
so that $H^{-} = a^{*} a$ and $H^{+} = a a^{*}$. Let now $\varphi \in C^{2} ( \R^{2}; \R )$ be a solution of the Poisson equation
\begin{equation} \label{1}
\Delta \varphi ( x ) = b ( x ) , \quad x \in \R^{2} .
\end{equation}
Then $A : = \big( - \frac{\partial \varphi}{\partial x_{2}} , \frac{\partial \varphi}{\partial x_{1}} \big)$ generates the magnetic field $b$, and moreover $\dive A = 0$. In this case we have
\begin{equation} \label{j3}
a^{*} = - 2 i e^{\varphi} \frac{\partial}{\partial z} e^{- \varphi} , \qquad a = - 2 i e^{-\varphi} \frac{\partial}{\partial \overline{z}} e^{\varphi} .
\end{equation}
Next, assume that the magnetic potentials $A^{( j )} \in C^{1} ( \R^{2} ; \R^{2})$, $j = 1 , 2$, generate the same magnetic field, i.e. $\curl A^{( 1 )} = \curl A^{( 2 )}$. Then there exists a function $\Phi \in C^{2} ( \R^{2} ; \R )$ such that $A^{( 1 )} = A^{( 2 )} + \nabla \Phi$. Therefore, $h ( A^{( 1 )} ) = e^{i \Phi} h ( A^{( 2 )} ) e^{- i \Phi}$ and, hence,
\begin{equation*}
H ( A^{( 1 )} , V ) = e^{i \Phi} H ( A^{(2)} , V ) e^{- i \Phi} ,
\end{equation*}
i.e. the operators $H ( A^{( 1 )} )$ and $H ( A^{( 2 )})$ are unitarily equivalent under the gauge transformation $u \mapsto e^{- i \Phi} u$. In particular, $H ( A^{( 1 )} )$ and $H ( A^{( 2 )})$ have identical spectral properties. The definition of the Pauli operator in arbitrary dimension $d \geq 2$, and the description of some of its basic spectral properties can be found, for example, in \cite{shi}.

\section{Ergodic properties of $H$} \label{s3}

In this section we consider the ergodic properties of the operator $H$ with almost periodic magnetic field $b$. We start with a brief summary of the definition of almost periodic functions and their basic properties following mainly \cite{shu2}. Since this part is independent of the dimension $d$, we let $d \geq 1$.

\subsection{Almost periodic functions}

Let $C_{\rm b} ( \R^{d} )$ be the non separable Banach space of bounded functions $f \in C ( \R^{d} )$ with norm
\begin{equation*}
\Vert f \Vert_{C_{\rm b} ( \R^{d} )} : = \sup_{x \in \R^{d}} \vert f ( x ) \vert .
\end{equation*}
Set
\begin{equation*}
e_{\lambda} ( x ) : = e^{i \lambda \cdot x} , \quad \lambda \in \R^{d} , \quad x \in\R^{d} .
\end{equation*}
Thus, $\{ e_{\lambda} \}_{\lambda \in \R^{d}}$ is the set of the continuous characters of the Abelian group $\R^{d}$. Put
\begin{equation*}
\Trig ( \R^{d} ) : = \bigg\{u = \sum_{j = 1}^{N} c_{j} e_{\lambda_{j}} \, | \, \ c_{j} \in \C , \ \lambda_{j} \in \R^{d} \text{ for } j = 1 , \ldots , N < \infty \bigg\} .
\end{equation*}
Then the Banach space of continuous almost periodic functions $\CAP( \R^{d} )$ is the closure of $\Trig ( \R^{d} )$ in $C_{\rm b}( \R^{d} )$. It is well known that if $f \in C_{\rm b} ( \R^{d} )$, then $f \in \CAP ( \R^{d} )$ if and only if the set $\{ f ( \cdot + s ) \}_{s \in \R^{d}}$ is precompact in $C_{\rm b} ( \R^{d} )$.

Let $f \in \CAP ( \R^{d} )$. Denote by
\begin{equation*}
\CM ( f ) : = \lim_{T \to \infty} T^{- d} \int_{( - T / 2 , T / 2)^{d}} f ( x ) \, d x \in \C ,
\end{equation*}
the mean value of $f$. For $\lambda \in \R^{d}$ denote by $f_{\lambda}$ {\em the Fourier coefficient}
\begin{equation*}
f_{\lambda} : = \CM ( f e_{- \lambda} ) ,
\end{equation*}
so that $f_{0} = \CM ( f )$, and put
\begin{equation*}
J ( f ) : = \big\{ \lambda \in \R^{d} \, | \, \ f_{\lambda} \neq 0 \big\} , \qquad J_{0} ( f ) : = J ( f ) \setminus \{ 0 \} .
\end{equation*}
It is well known that for any given $f \in \CAP ( \R^{d} )$, the set $J ( f )$ is countable, and $f$ is uniquely determined by the set $\{ f_{\lambda} \}_{\lambda \in \R^{d}}$. Let us note here the elementary fact that $f \in \CAP ( \R^{d} )$ is real valued if and only if $f_{- \lambda} = \overline{f_{\lambda}}$, $\lambda \in \R^{d}$.

We will need also the Wiener class of almost periodic functions
\begin{equation*}
\WAP ( \R^{d} ) : = \bigg\{ f \in \CAP ( \R^{d} ) \, | \,  \ \sum_{\lambda \in J ( f )} \vert f_{\lambda} \vert < \infty \bigg\} .
\end{equation*}
If $f \in \WAP ( \R^{d} )$, then $f$ coincides with the sum of its Fourier series
$\sum_{\lambda \in J ( f )} f_{\lambda} e_{\lambda} ( x )$,
which is absolutely convergent, uniformly with respect to $x \in \R^{d}$. Note also that if $f \in \WAP ( \R^{d} )$ and the set $J ( f )$ is bounded, then
\begin{equation*}
f \in \CAP^{\infty} ( \R^{d} ) : = \big\{ u \in C^{\infty} ( \R^{d} ) \, | \, \ D^{\alpha} u \in \CAP ( \R^{d} ) \text{ for } \alpha \in \Z_{+}^{d} \big\} .
\end{equation*}
Let $\B^{d}$ be the Bohr compactification of $\R^{d}$ (see e.g. \cite[Section 1]{shu2}). We recall that $\B^{d}$ is a compact Abelian group which is not metrizable and hence not first countable (see e.g. \cite[Remark 1.7 (b) and Theorem 1.3 (a)]{ctaw}). Further, there exists a continuous homomorphism $\iota : \R^{d} \to \B^{d}$ such that $\iota ( \R^{d} )$ is dense in $\B^{d}$. As in $\R^{d}$, we denote by ``$+$'' the group operation in $\B^{d}$. Note that $\iota$ induces an isomorphism between $\CAP ( \R^{d} )$ and $C ( \B^{d} )$. In particular, for each $f \in \CAP ( \R^{d} )$ there exists a unique $\phi \in C ( \B^{d} )$ such that
\begin{equation*}
f ( x ) = \phi ( \iota ( x ) ) , \quad x \in \R^{d} ;
\end{equation*}
we call $\phi$ {\em the canonic extension} of $f$. Let $\PPP$ be the Haar measure on $\B^{d}$, normalized to one, and $\CF$ be the $\sigma$-algebra of the $\PPP$-measurable subsets of $\B^{d}$. Then $( \B^{d} , \CF, \PPP )$ is a probability space. If $f \in \CAP ( \R^{d} )$ and $\phi \in C ( \B^{d} )$ is its canonic extension, then
\begin{equation*}
\CM ( f ) = \int_{\B^{d}} \phi ( \omega ) \, d \PPP ( \omega ) = : \E ( \phi ) .
\end{equation*}
Denote by $\epsilon_{\lambda}$ the canonic extension of $e_{\lambda}$, $\lambda \in \R^{d}$. Thus, $\{ \epsilon_{\lambda} \}_{\lambda \in \R^{d}}$ is the set of continuous characters of the group $\B^{d}$ which forms an orthonormal basis of $L^{2} ( \B^{d} , d \PPP )$.

\subsection{Operators with linear plus almost periodic magnetic potential} \label{s5}

Assume now that $b \in \CAP ( \R^{2} ; \R )$ and $V \in \CAP ( \R^{2} ; M_{2} )$. Recalling that $b_{0}$ is the mean value of $b$, set
\begin{equation*}
A_{0} : = \Big( - \frac{b_{0} x_{2}}{2} , \frac{b_{0} x_{1}}{2} \Big) ,
\end{equation*}
so that $\curl A_{0} = b_{0}$. Further, put $\widetilde{b} : = b - b_{0}$, and assume that there exists $\widetilde{A} \in \CAP ( \R^{2} ; \R^{2} )$ such that $\curl \widetilde{A} = \widetilde{b}$. This is for example the case if
\begin{equation} \label{2}
b(x) = b_{0} + \sum_{\lambda \in J_{0} ( b )} b_{\lambda} e_{\lambda} ( x ) , \quad x \in \R^{2} ,
\end{equation}
where $b_{\lambda} = \overline{b_{- \lambda}}$ for all $\lambda \in J ( b )$ and
\begin{equation} \label{3a}
\sum_{\lambda \in J_{0} ( b )} \vert b_{\lambda} \vert \big( 1 + \vert \lambda \vert^{- 1} \big) < \infty .
\end{equation}
Then $\widetilde{A}$ can be chosen in the form
\begin{equation} \label{5}
\widetilde{A} ( x ) = \bigg( i \sum_{\lambda \in J_{0} ( b )} b_{\lambda} \frac{\lambda_{2}}{\vert \lambda \vert^{2}} e_{\lambda} ( x ) , - i \sum_{\lambda \in J_{0} ( b )} b_{\lambda} \frac{\lambda_{1}}{\vert \lambda \vert^{2}} e_{\lambda} ( x ) \bigg) , \quad x \in \R^{2} .
\end{equation}
Thus, $\widetilde{A} \in \WAP ( \R^{2} ; \R^{2} )$ and $\curl \widetilde{A} = \widetilde{b}$. Eventually, we have $\curl A = b$ for $A : = A_{0} + \widetilde{A}$.

Let $\alpha \in C ( \B^{2} ; \R^{2} )$, $\beta \in C ( \B^{2} ; \R )$, and $\Upsilon \in C ( \B^{2} ; M_{2} )$ be the canonic extensions of $\widetilde{A}$, $b$ and $V$ respectively. Set
\begin{equation*}
\CA_{\omega} ( x ) : = \alpha ( \omega + \iota ( x ) ) , \qquad \CB_{\omega} ( x ) : = \beta ( \omega + \iota ( x ) ) , \qquad \CV_{\omega} ( x ) : = \Upsilon ( \omega + \iota ( x ) ) ,
\end{equation*}
for $x \in \R^{2}$ and $\omega \in \B^{2}$. On $\Dom h ( A_{0} )$ define the operators
\begin{equation} \label{j32}
\CH_{\omega}^{\pm} : = ( - i \nabla - A_{0} - \CA_{\omega} )^{2} \pm \CB_{\omega} ,
\end{equation}
self-adjoint in $L^{2} ( \R^{2} )$, and on $\Dom h ( A_{0} ) \oplus \Dom h ( A_{0} )$ define the operator
\begin{equation} \label{j33}
\CH_{\omega} = \left(
\begin{array} {cc}
\CH_{\omega}^{-} & 0\\
0 & \CH_{\omega}^{+}
\end{array}
\right) + \CV_{\omega} , \quad \omega \in \B^{2} ,
\end{equation}
self-adjoint in $L^{2} ( \R^{2} ; \C^{2} )$. Evidently, $\CH_{0} = H ( A , V )$. Note that the operator family $\{ \CH_{\omega} \}_{\omega \in \B^{2}}$ is continuous in the norm resolvent sense.

For $\xi \in \R^{2}$ introduce the unitary operators $\CU_{\xi} : L^{2} ( \R^{2} ) \to L^{2} ( \R^{2} )$ by
\begin{equation*}
( \CU_{\xi} f ) ( x ) = e^{i \frac{b_{0}}{2} ( \xi_{1} x_{2} - x_{1} \xi_{2} )} f ( x - \xi ) , \quad x \in \R^{2} , \quad f \in L^{2} ( \R^{2} ) .
\end{equation*}
We have
\begin{equation} \label{j28}
\CU_{\xi} ( - i \nabla - A_{0} ) \CU_{\xi}^{*} = - i \nabla - A_{0} , \qquad \CU_{\xi} \CB_{\omega} \CU_{\xi}^{*} = \CB_{\CT_{\xi} \omega} , \qquad \CU_{\xi} \CV_{\omega} \CU_{\xi}^{*} = \CV_{\CT_{\xi} \omega} ,
\end{equation}
and
\begin{equation} \label{j30}
\CU_{\xi} \CA_{\omega} \CU_{\xi}^{*} = \CA_{\CT_{\xi} \omega} ,
\end{equation}
where
\begin{equation} \label{j10a}
\CT_{\xi} \omega : = \omega - \iota ( \xi ) , \quad \omega \in \B^{2} , \quad \xi \in \R^{2} .
\end{equation}
Hence,
\begin{equation} \label{j10}
\CU_{\xi} \CH_{\omega} \CU_{\xi}^{*} = \CH_{\CT_{\xi} \omega} , \quad \omega \in \B^{2} , \quad \xi \in \R^{2} .
\end{equation}
We recall that a group of measure preserving automorphisms of $\B^{2}$, homomorphic to $\R^{2}$, is called $\R^{2}$-ergodic if any set $S \subset \CF$ invariant under the action of this group, satisfies either $\PPP ( S ) = 0$ or $\PPP ( S ) = 1$.

\begin{lem}\sl \label{l1}
The group $\{ \CT_{\xi} \}_{\xi \in \R^{d}}$ defined in \eqref{j10a} is $\R^{2}$-ergodic.
\end{lem}

\begin{proof} Due to the invariance of the Haar measure $\PPP$ under the action of $\B^{2}$, $\{ \CT_{\xi} \}_{\xi \in \R^{2}}$ is a group of measure preserving automorphisms. Assume that $S \in \CF$ is invariant under this group, i.e. $\CT_{\xi} S = S$ for all $\xi \in \R^{2}$.
Define the measure
\begin{equation*}
\mu_{S} ( C ) : = \PPP ( S \cap C ) , \quad C \in \CF.
\end{equation*}
We will show that this measure is invariant under the action of $\B^{2}$. Since $\PPP$ is a Haar measure on $\B^{2}$ and $S$ is invariant, we have
\begin{equation*}
\mu_{S} ( C + \iota ( \xi ) ) = \PPP ( S \cap ( C + \iota ( \xi ) ) ) = \PPP ( ( S - \iota ( \xi ) ) \cap C ) = \PPP ( S \cap C ) = \mu_{S} ( C ) ,
\end{equation*}
for $\xi \in \R^{2}$. From the continuity of the function $\B^{2} \ni \omega \mapsto \PPP ( S \cap ( C + \omega ) ) \in [ 0 , 1 ]$ and the density of $\iota ( \R^{2} )$ in $\B^{2}$, this yields
\begin{equation*}
\mu_{S} ( C + \omega ) = \mu_{S} ( C ) , \quad C \subset \CF , \quad \omega \in \B^{2} ,
\end{equation*}
i.e. the measure $\mu_{S}$ is invariant under $\B^{2}$. By the uniqueness property of the Haar measure, there exists a constant $a = a ( S ) \geq 0$ such that
\begin{equation*}
\mu_{S} ( C ) = a \PPP ( C ) , \quad C \in \CF .
\end{equation*}
If $a = 0$, then $\PPP ( S ) = \mu_{S} ( S ) = 0$. If $a > 0$ then
\begin{equation*}
\PPP ( S ) = \frac{\mu_{S} ( S )}{a} = \frac{\mu_{S} ( \B^{2} )}{a} = \PPP ( \B^{2} ) = 1 .
\end{equation*}
Hence, $\{ \CT_{\xi} \}_{\xi \in \R^{2}}$ is an $\R^{2}$-ergodic group of automorphisms.
\end{proof}

\begin{xrem}
Various versions of Lemma \ref{l1} are available in the literature (see e.g. \cite[Section 10.1]{cfks} for a somewhat different but closely related situation). We include its proof just for the sake of completeness of the exposition.
\end{xrem}

Using standard properties of ergodic operator families (see e.g. \cite{km, k, pf}), we obtain the following

\begin{thm}\sl \label{th1}
Let $b \in \CAP ( \R^{2} ; \R )$ and $V \in \CAP ( \R^{2} ; M_{2} )$ be such that there exists $\widetilde{A} \in \CAP ( \R^{2} ; \R^{2} )$ with $\curl ( A_{0} + \widetilde{A} ) = b$. Then, \newline
{\rm (i)} There exist closed subsets $\Sigma$, $\Sigma_{\rm ac}$, $\Sigma_{\rm sc}$ and $\Sigma_{\rm pp}$ of $\R$ such that $\PPP$-almost surely
\begin{equation*}
\sigma ( \CH_{\omega} ) = \Sigma, \qquad \sigma_{\rm ac} ( \CH_{\omega} ) = \Sigma_{\rm ac} , \qquad \sigma_{\rm sc} ( \CH_{\omega} ) = \Sigma_{\rm sc} , \qquad \sigma_{\rm pp} ( \CH_{\omega} ) = \Sigma_{\rm pp} .
\end{equation*}
{\rm (ii)} Moreover, $\PPP$-almost surely
\begin{equation*}
\sigma_{\rm disc} ( \CH_{\omega} ) = \emptyset .
\end{equation*}
{\rm (iii)} Any $E \in \R$ is $\PPP$-almost surely not an eigenvalue of $\CH_{\omega}$ of finite multiplicity.
\end{thm}

\begin{xrem}
In the case $V = 0$, the operator family $\{ \CH_{\omega} \}_{\omega \in \B^{2}}$ was introduced in the proof of \cite[Lemma 3.2]{r5}. There are many works considering the non-magnetic Schr\"odinger operator $- \Delta + V$ acting in $L^{2} ( \R^{d} )$, $d \geq 1$, with almost periodic scalar potential $V$. However, usually the corresponding ergodic family $- \Delta + \CV_{\omega}$ is defined for $\omega$ on {\em the hull} of $V$ (see e.g \cite{pf, as}). Our choice to define the ergodic family $\CH_{\omega}$ for $\omega \in \B^{2}$ is motivated by the fact that there are several scalar almost periodic functions involved in $H ( A , V )$ and, on the other hand, the hull of any scalar function $f \in \CAP ( \R^{2} )$ is a subgroup of $\B^{2}$.
\end{xrem}

Applying a suitable continuity argument, we show in Corollary \ref{f1a} below that the results of Theorem \ref{th1} concerning $\sigma ( \CH_{\omega} )$ and $\sigma_{\rm disc} ( \CH_{\omega} )$ hold for {\em every} $\omega \in \B^{2}$.

\begin{cor}\sl \label{f1a}
Under the hypotheses of Theorem \ref{th1}, we have
\begin{gather}
\sigma ( \CH_{\omega} ) = \Sigma ,   \label{j23} \\
\sigma_{\rm disc} ( \CH_{\omega} ) = \emptyset ,  \label{j24}
\end{gather}
for any $\omega \in \B^{2}$. In particular, $\sigma_{\rm disc} ( H ) = \sigma_{\rm disc} ( \CH_{0} ) = \emptyset$.
\end{cor}

\begin{proof}
Since $\PPP ( U ) > 0$ for any open non-empty $U \subset \B^{2}$, every $S \in \CF$ with $\PPP (S) = 1$ is dense in $\B^{2}$. Set
\begin{equation*}
S_{0} : = \{ \omega \in \B^{2} \, | \, \sigma ( \CH_{\omega} ) = \Sigma \text{ and } \sigma_{\rm disc} ( \CH_{\omega} ) = \emptyset \} .
\end{equation*}
By Theorem \ref{th1}, we have $\PPP ( S_{0} ) = 1$, and then $S_{0}$ is dense in $\B^{2}$. Let $\omega \in \bd$. Pick a net $\{ \omega_{\alpha} \}_{\alpha \in \CI} \subset S_{0}$ which converges to $\omega$. Then we have
\begin{equation} \label{j20}
( \CH_{\omega_{\alpha}} - z )^{- 1} \longrightarrow ( \CH_{\omega} - z )^{- 1} , \quad z \in \C \setminus \Big[ - \sup_{x \in \R^{2}} \vert V ( x ) \vert , \infty \Big) ,
\end{equation}
in norm. The rest of the proof is based on standard perturbation arguments. We include some details just because we're dealing with operator nets instead of operator sequences.

From the spectral theorem, the resolvent of a self-adjoint operator $T$ satisfies
\begin{equation} \label{k1}
\Vert ( T - z )^{- 1} \Vert = \frac{1}{\dist ( z , \sigma ( T ) )} , \quad z \in \C \setminus \R .
\end{equation}
Using \eqref{j20}, \eqref{k1} and $\sigma ( \CH_{\omega_{\alpha}} ) = \Sigma$ for all $\alpha \in \CI$, we deduce
\begin{equation*}
\dist ( z , \sigma ( \CH_{\omega} ) ) = \dist ( z , \Sigma ) , \quad z \in \C \setminus \R .
\end{equation*}
Since $\Sigma$ and $\sigma ( \CH_{\omega_{\alpha}} )$ are closed subsets of $\R$, this easily implies \eqref{j23}.

Now assume that there exists $E \in \sigma_{\rm disc} ( \CH_{\omega} )$. Then, \eqref{j23} shows that $E$ is an isolated eigenvalue of infinite multiplicity of $\CH_{\omega_{\alpha}}$ for any $\alpha \in \CI$. Now pick $\varepsilon > 0$ such that $( E - \varepsilon , E + \varepsilon ) \cap \sigma ( \CH_{\omega} ) = \{ E \}$. Passing to resolvents and applying  \cite[Chapter 9, Section 4, Lemma 3]{bs}, we find that there exists $\beta \in \CI$ such that $\alpha \geq \beta$ implies
\begin{equation} \label{j22}
\Tr \one_{( E - \varepsilon , E + \varepsilon )} ( \CH_{\omega} ) \geq \Tr \one_{( E - \varepsilon / 2 , E + \varepsilon / 2 )} ( \CH_{\omega_{\alpha}} ) .
\end{equation}
Here $\one_{S} ( T )$ denotes the spectral projection of the self-adjoint operator $T$ corresponding to the Borel set $S \subset \R$. Since the l.h.s. of \eqref{j22} is finite and its r.h.s. is infinite, we obtain a contradiction which gives \eqref{j24}.
\end{proof}

\begin{xrem}\rm
An analogue of Corollary \ref{f1a} for the case of  $- \Delta + V$ in $L^{2} ( \R^{d} )$, $d \geq 1$, with almost periodic $V$, is contained in \cite[Theorems A.2.1, A.2.2]{as}. As already mentioned, in \cite{as}, the operator family $- \Delta + \CV_{\omega}$ is defined on the hull of $V$.
\end{xrem}

\subsection{Operators with general almost periodic magnetic fields} \label{s6}

Our next goal is to investigate the ergodic properties of the operator $H ( A , V )$ assuming only that $b \in \CAP ( \R^{2} ; \R )$ and $V \in \CAP ( \R^{2} ; M_{2} )$ but not that there exists $\widetilde{A} \in \CAP ( \R^{2} ; \R^{2} )$ which generates $\widetilde{b} = b - b_{0}$. In fact, we will suppose a bit more about $b$, namely that $b \in \WAP ( \R^{2} ; \R )$. Then we have
\begin{equation*}
\CB_{\omega} ( x ) = b_{0} + \sum_{\lambda \in J_{0} ( b )} b_{\lambda} \eps_{\lambda} ( \omega ) e_{\lambda}( x ) , \quad \omega \in \B^{2} , \quad x \in \R^{2} .
\end{equation*}
For $\omega \in \B^{2}$ and $x \in \R^{2}$, set
\begin{equation*}
\SA_{\omega} ( x ) : = \bigg( i \sum_{\lambda \in J_{0} ( b )} b_{\lambda} \frac{\lambda_{2} \epsilon_{\lambda} ( \omega )}{\vert \lambda \vert^{2}{}}(e_{\lambda} ( x ) - 1 ) , - i \sum_{\lambda \in J_{0} ( b )} b_{\lambda} \frac{\lambda_{1} \epsilon_{\lambda} ( \omega )}{\vert \lambda \vert^{2}}(e_{\lambda} ( x ) - 1 ) \bigg) .
\end{equation*}
Since we do not assume any more that \eqref{3a} is true, generically $\SA_{\omega} \notin \CAP(\R^{2}; \R^{2})$. However, the series above converge absolutely, uniformly with respect to $\omega \in \B^{2}$, and locally uniformly with respect to $x \in \R^{2}$. It is easy to see that  $\SA_{\omega} \in C^{1}(\R^{2}; \R^{2})$ and $\curl \SA_{\omega} = \CB_{\omega} - b_{0}$ for each $\omega \in \B^{2}$. Moreover, $\vert e_{\lambda} ( x ) - 1 \vert \leq \vert \lambda \vert \vert x \vert$ gives
\begin{equation} \label{au4}
\vert \SA_{\omega} ( x ) \vert \leq \sqrt{2} \bigg( \sum_{\lambda \in J_{0} ( b )} \vert b_{\lambda} \vert \bigg) \vert x \vert , \quad \omega \in \B^{2} , \quad x \in \R^{2} ,
\end{equation}
and, using in addition that $\vert e_{\lambda} ( x ) - 1 \vert \leq 2$, we get
\begin{equation} \label{au30}
\lim_{\vert x \vert \to \infty} \vert x \vert^{-1} \SA_{\omega} ( x ) = 0 .
\end{equation}
Similarly to \eqref{j32}--\eqref{j33}, set
\begin{equation*}
\SH_{\omega}^{\pm} : = ( - i \nabla - A_{0} - \SA_{\omega} )^{2} \pm \CB_{\omega} , \qquad \SH_{\omega} =
\left( \begin{array}{cc}
\SH_{\omega}^{-} & 0\\
0 & \SH_{\omega}^{+}
\end{array} \right) + \CV_{\omega}, \quad \omega \in \B^{2} .
\end{equation*}
Again, $\SH_{0} = H ( A_{0} + \SA_{0} , V )$.

\begin{proposition}\sl \label{pau1}
Let $z \in \C \setminus \R$ and $f \in L^{2} ( \R^{2} ; \C^{2} )$. Then the function
\begin{equation*}
\B^{2} \ni \omega \longmapsto ( \SH_{\omega} - z )^{- 1} f \in L^{2} ( \R^{2} ; \C^{2} ) ,
\end{equation*}
is continuous.
\end{proposition}

\begin{proof}
Pick $\omega \in \B^{2}$ and a net $\{ \omega_{\alpha} \} \subset \B^{2}$ such that $\omega_{\alpha} \to \omega$. We will show that
\begin{equation} \label{au8}
\big\Vert \big( ( \SH_{\omega_{\alpha}} - z )^{- 1} - ( \SH_{\omega} - z )^{- 1} \big) f \big\Vert_{L^{2} ( \R^{2} ; \C^{2} )} \longrightarrow 0 .
\end{equation}
Since
\begin{equation*}
\big\Vert \big( ( \SH_{\omega_{\alpha}} - z )^{- 1} - ( \SH_{\omega} - z )^{- 1} \big) f \big\Vert_{L^{2} ( \R^{2} ; \C^{2} )} \leq \frac{2}{\vert \im z \vert} \Vert f \Vert_{L^{2} ( \R^{2} ; \C^{2} )} ,
\end{equation*}
and $C_{0}^{\infty} ( \R^{2} ; \C^{2} )$ is dense in $L^{2} ( \R^{2} ; \C^{2} )$, we can assume without loss of generality that $f \in C_{0}^{\infty} ( \R^{2} ; \C^{2} )$.
Further,
\begin{align}
\big( ( \SH_{\omega_{\alpha}} & - z )^{- 1} - ( \SH_{\omega} - z )^{- 1} \big) f   \nonumber  \\
&\!\!\!= ( \SH_{\omega_{\alpha}} - z )^{- 1} \left( \left(
\begin{array} {cc}
\SH_{\omega}^{-} - \SH_{\omega_{\alpha}}^{-} & 0  \\
0 & \SH_{\omega}^{+} - \SH_{\omega_{\alpha}}^{+}
\end{array}
\right) + \CV_{\omega} - \CV_{\omega_{\alpha}} \right) ( \SH_{\omega} - z )^{- 1} f .  \label{au9}
\end{align}
It is easy to check that
\begin{equation} \label{au1}
\big\Vert ( \SH_{\omega_{\alpha}} - z )^{-1} ( \CV_{\omega} - \CV_{\omega_{\alpha}} ) ( \SH_{\omega} - z )^{- 1} \big\Vert \leq \frac{\Vert \CV_{\omega} - \CV_{\omega_{\alpha}} \Vert}{\vert \im z \vert^{2}} \longrightarrow 0 ,
\end{equation}
as $\omega_{\alpha} \to \omega$. Next, write
\begin{align}
( \SH_{\omega_{\alpha}} - z & )^{- 1}
\left(
\begin{array} {cc}
\SH_{\omega}^{-} - \SH_{\omega_{\alpha}}^{-} & 0  \\
0 & \SH_{\omega}^{+} - \SH_{\omega_{\alpha}}^{+}
\end{array}
\right) ( \SH_{\omega} - z )^{- 1} f  \nonumber \\
={}& \SC_{\omega_{\alpha}} ( \overline{z} )^{*} \SD ( \omega , \omega_{\alpha} )^{*} \< x \> ( \SH_{\omega} - z )^{- 1} f + ( \SH_{\omega_{\alpha}} - z )^{- 1} \SD ( \omega , \omega_{\alpha} ) \< x \> \SC_{\omega} ( z ) f  \nonumber  \\
={}& \SC_{\omega_{\alpha}} ( \overline{z})^{*} \SD ( \omega , \omega_{\alpha} )^{*} ( \SH_{\omega} - z)^{- 1} \< x \> f + \SC_{\omega_{\alpha}} ( \overline{z})^{*} \SD ( \omega, \omega_{\alpha} )^{*} \big[ \< x \> , ( \SH_{\omega} - z )^{- 1} \big] f \nonumber \\
&+ ( \SH_{\omega_{\alpha}} - z )^{- 1} \SD ( \omega, \omega_{\alpha} ) \SC_{\omega} ( z ) \< x \> f + ( \SH_{\omega_{\alpha}} - z )^{- 1} \SD ( \omega , \omega_{\alpha} ) \big[ \< x \> , \SC_{\omega} ( z ) \big] f , \label{au6}
\end{align}
where
\begin{equation*}
\SC_{\nu} ( z ) : =
\left(
\begin{array}{cc}
a_{\nu} & 0\\
0 & a_{\nu}^{*}
\end{array}
\right) ( \SH_{\nu} - z )^{- 1} ,
\end{equation*}
with
\begin{equation*}
a_{\nu} : = a ( A_{0} + \SA_{\nu} ) , \quad \nu \in \B^{2} ,
\end{equation*}
the operator $a ( A )$ being defined in \eqref{au3},
\begin{equation*}
\SD ( \omega , \omega_{\alpha} ) : = \left(
\begin{array}{cc}
( a_{\omega}^{*} - a_{\omega_{\alpha}}^{*} ) \< x \>^{- 1} & 0 \\
0 & ( a_{\omega} - a_{\omega_{\alpha}} ) \< x \>^{- 1}
\end{array}
\right) ,
\end{equation*}
which is a multiplication operator, and $\< x \> : = ( 1 + \vert x \vert^{2} )^{1 / 2}$. Similarly to \eqref{au4},
\begin{equation} \label{au5}
\Vert \SD ( \omega , \omega_{\alpha} ) \Vert \leq \sqrt{2} \sum_{\lambda \in J_{0} ( b )} \vert b_{\lambda} \vert \, \vert \epsilon_{\lambda} ( \omega ) - \epsilon_{\lambda} ( \omega_{\alpha} ) \vert .
\end{equation}
Since $\sum_{\lambda \in J_{0} ( b )} \vert b_{\lambda} \vert < \infty$, $\vert \epsilon_{\lambda} ( \omega ) - \epsilon_{\lambda} ( \omega_{\alpha} ) \vert \leq 2$ for all $\lambda \in \R^{2}$ and $\vert \epsilon_{\lambda} ( \omega ) - \epsilon_{\lambda} ( \omega_{\alpha} ) \vert \to 0$ as $\omega_{\alpha} \to \omega$ for any fixed $\lambda \in \R^{2}$, we find that \eqref{au5} implies
\begin{equation} \label{au10}
\Vert \SD ( \omega , \omega_{\alpha} ) \Vert \longrightarrow 0 ,
\end{equation}
as $\omega_{\alpha} \to \omega$. Further, we have
\begin{align}
\Vert \SC_{\nu} ( z ) \Vert^{2} &= \Vert \SC_{\nu} ( z )^{*} \SC_{\nu} ( z ) \Vert = \big\Vert ( \SH_{\nu} - \overline{z} )^{- 1} ( \SH_{\nu} - \CV_{\nu} ) ( \SH_{\nu} - z )^{- 1} \big\Vert   \nonumber \\
&\leq \big\Vert ( \SH_{\nu} - \overline{z} )^{- 1} \big\Vert + \big\Vert ( \SH_{\nu} - \overline{z} )^{- 1} ( z - \CV_{\nu} ) ( \SH_{\nu} - z )^{- 1} \big\Vert  \nonumber \\
&\leq \frac{1}{\vert \im z \vert} + \frac{\vert z \vert + \sup_{x \in \R^{2}} \vert V ( x ) \vert}{\vert \im z \vert^{2}} ,
\end{align}
for $\nu \in \B^{2}$. Since the gradient of $\< x \>$ is bounded and
\begin{align*}
\big[ \< x & \> , ( \SH_{\nu} - z )^{- 1} \big] \\
&= ( \SH_{\nu} - z )^{- 1} \left(
\begin{array}{cc}
a_{\nu}^{*} [ a_{\nu} , \< x \> ] + [ a_{\nu}^{*} , \< x \> ] a_{\nu} & 0 \\
0 & a_{\nu} [ a_{\nu}^{*} , \< x \> ] + [ a_{\nu} , \< x \> ] a_{\nu}^{*}
\end{array}
\right) ( \SH_{\nu} - z )^{- 1} ,
\end{align*}
we find that
\begin{equation}
\big\Vert \big[ \< x \> , ( \SH_{\nu} - z )^{- 1} \big] \big\Vert \leq 2 \big\Vert ( \SH_{\nu}  - z )^{- 1} \big\Vert \Vert [ a_{\nu} , \< x \> ] \Vert \big( \Vert \SC_{\nu} ( z ) \Vert + \Vert \SC_{\nu} ( \overline{z} ) \Vert \big) \leq C ,
\end{equation}
and, analogously,
\begin{equation} \label{au12}
\big\Vert \big[ \< x \> , \SC_{\nu} ( z ) \big] \big\Vert \leq C , \quad \nu \in \B^{2} ,
\end{equation}
with a constant $C>0$ which may depend on $z$ but is independent of $\nu$.
Since ${\rm supp}\,f$ is compact, putting together \eqref{au10}--\eqref{au12}, we find that \eqref{au6} implies
\begin{equation*}
\bigg\Vert ( \SH_{\omega_{\alpha}} - z )^{- 1} \left(
\begin{array}{cc}
\SH_{\omega}^{-} - \SH_{\omega_{\alpha}}^{-} & 0 \\
0 & \SH_{\omega}^{+} - \SH_{\omega_{\alpha}}^{+}
\end{array}
\right) ( \SH_{\omega} - z )^{- 1} f \bigg\Vert_{L^{2} ( \R^{2} ; \C^{2} )} \longrightarrow 0 ,
\end{equation*}
as $\omega_{\alpha} \to \omega$, which combined with \eqref{au9} and \eqref{au1}, yields \eqref{au8}.
\end{proof}

\begin{cor}\sl \label{fau1}
For $\lambda \in \R$, the operator family $\{ \one_{( - \infty , \lambda )} ( \SH_{\omega} ) \}_{\omega \in \B^{2}}$ is weakly measurable, i.e. the functions
\begin{equation} \label{au13}
\B^{2} \ni \omega \longmapsto \big\< \one_{( - \infty , \lambda )} ( \SH_{\omega} ) f , g \big\> \in \C , \quad f , g \in L^{2} ( \R^{2} ; \C^{2} ) ,
\end{equation}
are $\PPP$-measurable, $\< \cdot , \cdot \>$ being the scalar product in $L^{2} ( \R^{2} ; \C^{2} )$.
\end{cor}

\begin{proof}
Let $z \in \C \setminus \R$. It follows from Proposition \ref{pau1}, that the functions
\begin{equation*}
\B^{2} \ni \omega \longmapsto \big\< ( \SH_{\omega}  - z )^{- 1} f , g \big\> \in \C , \quad f , g \in L^{2} ( \R^{2} ; \C^{2} ) ,
\end{equation*}
are continuous, and hence $\PPP$-measurable. By \cite[Proposition 3]{km}, this is equivalent to the measurability of the functions defined in \eqref{au13}.
\end{proof}

Note that \eqref{j28} remains unchanged but \eqref{j30} should be replaced by
\begin{equation*}
\CU_{\xi} \SA_{\omega} \CU_{\xi}^{*} = \SA_{\CT_{\xi} \omega} + \nabla \Phi_{\omega , \xi} ,
\end{equation*}
where
\begin{equation*}
\Phi_{\omega , \xi} ( x ) = x \cdot \SA_{\omega} ( - \xi) , \quad x \in \R^{2} , \quad \xi \in \R^{2} , \quad \omega \in \R^{2} .
\end{equation*}
Hence, \eqref{j10} should be replaced by
\begin{equation} \label{7a}
\CU_{\xi} \SH_{\omega} \CU_{\xi}^{*} = e^{-i \Phi_{\xi , \omega}} \SH_{\CT_{\xi} \omega} e^{i \Phi_{\xi , \omega}} , \quad \omega \in \B^{2} , \quad \xi \in \R^{2} .
\end{equation}
Thus the operator family $\{\SH_{\omega} \}_{\omega \in \B^{2}}$ is not any more ergodic in the classical sense (see \eqref{j10}), but is ergodic up to an $\omega$-dependent gauge transformation. However, relation \eqref{7a} defines a reasonable generalization of the $\R^{2}$-ergodicity, allowing us to prove Theorem \ref{th2} below, thus extending Theorem \ref{th1} to the case where we just assume almost periodicity of $b$ and $V$.

\begin{thm}\sl \label{th2}
Assume that $b \in \WAP ( \R^{2} ; \R )$ and $V \in \CAP ( \R^{2} ; M_{2} )$.
Then,

\noindent
{\rm (i)} There exist closed subsets $\Sigma$, $\Sigma_{\rm ac}$, $\Sigma_{\rm sc}$, and $\Sigma_{\rm pp}$ of $\R$ such that  $\PPP$-almost surely
\begin{equation*}
\sigma ( \SH_{\omega} ) = \Sigma , \qquad \sigma_{\rm ac} ( \SH_{\omega} ) = \Sigma_{\rm ac} ,
\qquad \sigma_{\rm sc} ( \SH_{\omega} ) = \Sigma_{\rm sc} , \qquad \sigma_{\rm pp} ( \SH_{\omega} ) = \Sigma_{\rm pp} .
\end{equation*}

\noindent
{\rm (ii)} Moreover, $\PPP$-almost surely,
\begin{equation*}
\sigma_{\rm disc} ( \SH_{\omega} ) = \emptyset .
\end{equation*}

\noindent
{\rm (iii)} Any $E \in \R$ is $\PPP$-almost surely not an eigenvalue of $\SH_{\omega}$ of finite multiplicity.
\end{thm}

\noindent
In the proof of the theorem we will need Lemma \ref{lau1} below whose first (resp. second) part is very close to Proposition 5 (resp.  Proposition 6) of \cite[Chapter 4]{k}.

\begin{lem}\sl \label{lau1}
Let $\{ P_{\omega} \}_{\omega \in \B^{2}}$ be a weakly measurable family of orthogonal projections acting in $L^{2} ( \R^{2} ; \C^{2} )$, which satisfies
\begin{equation} \label{au20}
\CU_{\xi} P_{\omega} \CU_{\xi}^{*} = e^{- i \Phi_{\xi , \omega}} P_{\CT_{\xi} \omega} e^{i \Phi_{\xi , \omega}} , \quad \omega \in \B^{2} , \quad \xi \in \R^{2} .
\end{equation}
Then,

\noindent
{\rm (i)} The function
\begin{equation} \label{au20a}
\B^{2} \ni \omega \longmapsto \Tr P_{\omega} \in \Z_{+} \cup \{ \infty \} ,
\end{equation}
is almost surely constant.

\noindent
{\rm (ii)} Either $\Tr P_{\omega} = 0$ almost surely or $\Tr P_{\omega} = \infty$ almost surely.
\end{lem}

\begin{proof} (i) By the weak measurability of $P_{\omega}$, the function defined in \eqref{au20a} is $\PPP$-measurable. By \eqref{au20}, we have
\begin{equation*}
\Tr P_{\omega} = \Tr P_{\CT_{\xi} \omega} , \quad \xi \in \R^{2} , \quad \omega \in \B^{2} ,
\end{equation*}
i.e. $\Tr P_{\omega}$ is invariant under the action of the ergodic group $\{ \CT_{\xi} \}_{\xi \in \R^{2}}$. By \cite[Proposition 9.1]{cfks}, $\Tr P_{\omega}$ is almost surely constant.

(ii) By Part (i), there exists $n \in \Z_{+} \cup \{ \infty \}$ such that almost surely
\begin{equation} \label{au22}
\Tr P_{\omega} = \E ( \Tr P_{\omega} ) = n .
\end{equation}
Hence, we must just exclude the possibility that $n \in \N$ in \eqref{au22}. Assume that \eqref{au22} holds true with $n \in \N$. Since $n > 0$, we find that for any total set $X \subset L^{2} ( \R^{2} ; \C^{2} )$ there exists an element $\phi \in X$ such that
\begin{equation} \label{au23}
\E ( \< P_{\omega} \phi , \phi \> ) > 0 .
\end{equation}
Define
\begin{equation*}
X_{0} : = \big\{ \phi \in L^{2} ( \R^{2} ; \C^{2} ) \, | \, \ \supp \phi \subset ( - 1 / 2 , 1 / 2 )^{2} + x \text{ for some } x \in \R^{2} \big\} ,
\end{equation*}
where $\< \cdot ,\cdot \>$ is the scalar product in $L^{2} ( \R^{2} ; \C^{2} )$. Evidently, $X_{0}$ is total in $L^{2} ( \R^{2} ; \C^{2} )$. Pick $\phi \in X_{0}$ with $\< \phi , \phi \> = 1$ such that \eqref{au23} holds true. Note that the system $\{ \CU^{*}_{\xi} e^{- i \Phi_{\xi,\omega}} \phi \}_{\xi \in \Z^{2}}$ is orthonormal in $L^{2} ( \R^{2} ; \C^{2} )$. Therefore,
\begin{equation} \label{au24}
\Tr P_{\omega} = \E ( \Tr P_{\omega} ) \geq \sum_{\xi \in \Z^{2}} \E \big( \big\< P_{\omega} \CU^{*}_{\xi} e^{- i \Phi_{\xi , \omega}} \phi , \CU^{*}_{\xi} e^{- i \Phi_{\xi , \omega}} , \phi \big\> \big) = \sum_{\xi \in \Z^{2}} \E ( \< P_{\CT_{\xi} \omega} \phi , \phi \> ) .
\end{equation}
Since the transformations $\CT_{\xi}$  are measure preserving, we have
\begin{equation} \label{au25}
\E ( \< P_{\CT_{\xi} \omega } \phi , \phi \> ) = \E ( \< P_{\omega} \phi , \phi \> ) , \quad \xi \in \R^{2} .
\end{equation}
By \eqref{au23} and \eqref{au25}, we find that the rightmost term in \eqref{au24} is infinite which contradicts our assumption that $n$ in \eqref{au22} is finite.
\end{proof}

\begin{proof}[Proof of Theorem \ref{th2}]
Fix $\lambda \in \R$. Then \eqref{7a} implies that equality \eqref{au20} holds true for the family $P_{\omega} = \one_{( - \infty , \lambda )} ( \SH_{\omega} )$, $\omega \in \B^{2}$, which is weakly measurable by Corollary \ref{fau1}. Thus we find that $\Tr \one_{( - \infty , \lambda)} ( \SH_{\omega} )$ is almost surely constant, which implies that
$\sigma ( \SH_{\omega} )$, $\sigma_{\rm ess} ( \SH_{\omega} )$ and $\sigma_{\rm disc} ( \SH_{\omega} )$ are almost surely constant.

Let $\SH_{\omega}^{\rm ac}$, $\SH_{\omega}^{\rm sc}$ and $\SH_{\omega}^{\rm pp}$ be the absolute continuous, the singular continuous and pure point part of $\SH_{\omega}$ respectively. By \eqref{7a}, equality \eqref{au20} holds true for $P_{\omega} = \one_{( - \infty , \lambda )} ( \SH_{\omega}^{\sharp} )$ with $\sharp = {\rm ac , sc , pp}$. By \cite[Section 3]{km}, Corollary \ref{fau1} implies that the family of operators $P_{\omega} = \one_{( - \infty , \lambda )} ( \SH_{\omega}^{\sharp} )$
is weakly measurable, and, hence, by Lemma \ref{lau1} $\Tr \one_{( - \infty , \lambda )} ( \SH_{\omega}^{\sharp})$, is almost surely constant, which implies, as above, that $\sigma_{\sharp} ( \SH_{\omega} )$, $\sharp = {\rm ac , sc , pp}$, is almost surely constant. This concludes the proof of Theorem \ref{th2} (i).

The remaining two parts of Theorem \ref{th2} now follow from Lemma \ref{lau1} just as Theorem 3 and Corollary 1  in \cite[Section 4]{k} follow from Propositions 5 and 6 there.
\end{proof}

\begin{xrem}
(i) A result closely related to Theorem \ref{th2} is contained in \cite[Theorem 2.1]{ueki}. It is possible that our theorem could be deduced from that result which however is fairly abstract, so we preferred to provide a relatively simple and self-contained independent proof.

(ii) Theorem \ref{th1}, Corollary \ref{f1a} and Theorem \ref{th2} admit straightforward but quite technical generalizations to the case $d \geq 3$. We omit them since we don't believe that they would add a new and deeper insight to the problems considered.

(iii) For the moment, we do not know whether an analogue of Corollary \ref{f1a} holds true under the hypotheses of Theorem \ref{th2}.
\end{xrem}

\section{Zero modes of $H$} \label{s4}

In this section we assume $V = 0$ and write $H$ instead of $H ( A , 0)$. We are interested in the zero modes of the positive operator $H$, i.e. in the closed subspace $\Ker H$ of the Hilbert space $L^{2} ( \R^{2} ; \C^{2} )$. From \eqref{j2}, we have
\begin{equation} \label{a}
\Ker H = \big\{ {\bf u} = ( u_{1} , u_{2} ) \, | \, \ u_{1} \in \Ker a , \ u_{2} \in \Ker a^{*} \big\} ,
\end{equation}
and, hence,
\begin{equation} \label{aa}
\dim \Ker H = \dim \Ker a + \dim \Ker a^{*} .
\end{equation}
Moreover, \eqref{j3} yields
\begin{gather}
\Ker a = \Big\{ u \in L^{2} ( \R^{2} ) \, | \, \ u = f e^{-\varphi} , \ \frac{\partial f}{\partial \overline{z}} = 0 \Big\} ,  \label{b} \\
\Ker a^{*} = \Big\{u \in L^{2} ( \R^{2} ) \, | \, \ u = f e^{\varphi} , \ \frac{\partial f}{\partial z} = 0 \Big\} .   \label{c}
\end{gather}

\subsection{Classical results}  \label{ss41}
Let us mention two classes of magnetic fields $b$ for which $\Ker H$ is well described.

\subsubsection{Rapidly decaying magnetic fields} \label{412}

Assume that $b \in C^{\infty} ( \R^{2} ; \R )$ satisfies
\begin{equation*}
\vert b ( x ) \vert \leq C \<x\>^{- 2 - \varepsilon} , \quad x \in \R^{2} ,
\end{equation*}
with $C , \varepsilon > 0$. Then the function
\begin{equation*}
\varphi ( x ) = \frac{1}{2 \pi} \int_{\R^{2}} \ln \vert x - y \vert b ( y ) \, d y , \quad x \in \R^{2} ,
\end{equation*}
is well defined and is a solution of \eqref{1}. Moreover, we have
\begin{equation*}
\varphi ( x ) = \Phi \ln \vert x \vert + o ( 1 ), \quad \vert x \vert \to \infty ,
\end{equation*}
where $\Phi : = \frac{1}{2 \pi} \int_{\R^{2}} b ( y ) \, d y$ is {\em the flux of the magnetic field}. As a result,
\begin{equation} \label{au31a}
\dim \Ker H = \lfloor \vert \Phi \vert \rfloor ,
\end{equation}
where $\lfloor t \rfloor$ is the greatest integer less than $t > 0$, and $\lfloor 0 \rfloor = 0$ (see the original work \cite{ac} or \cite[Theorem 6.5]{cfks}). Moreover, since $\sigma ( H )$ is purely essential and coincides with $[ 0 , \infty )$ (see e.g. \cite[Theorem 6.1]{cfks}), the zero eigenvalue of $H$ is the endpoint of its essential spectrum.

\begin{xrem}
Relation \eqref{au31a}, known as {\em the Aharonov-Casher theorem}, has been generalized in various directions during the last two decades \cite{ev, gs, rs, e}.
\end{xrem}

\subsubsection{Periodic magnetic fields} \label{413}

Suppose now that
\begin{equation*}
b ( x ) = \sum_{k \in \Z^{2}} b_{k} e^{i k \cdot x} , \quad x \in \R^{2} ,
\end{equation*}
with $\{ b_{k} \}_{k \in \Z^{2}} \in \ell^{1} ( \Z^{2} )$ and $\overline{b_{k}} = b_{- k}$ for $k \in \Z^{2}$. In particular, $b \in C ( \T^{2} ; \R )$. We can choose the solution $\varphi$ in \eqref{1} as $\varphi = \varphi_{0} + \widetilde{\varphi}$ where
\begin{equation*}
\varphi_{0} ( x ) : =  \frac{b_{0} \vert x \vert^{2}}{4} ,
\qquad
\widetilde{\varphi} ( x ) : = -\sum_{0 \neq k \in \Z^{2}} \frac{b_{k}}{\vert k \vert^{2}} e^{i k \cdot x}, \quad x \in \R^{2} .
\end{equation*}
Then $\widetilde{\varphi}$ is real, bounded and we have
\begin{equation} \label{au32}
\varphi ( x ) : =  \frac{b_{0} \vert x \vert^{2}}{4} + \CO ( 1 ) , \quad x \in \R^{2} .
\end{equation}
Hence, \eqref{b}--\eqref{c} easily imply
\begin{equation*}
\dim \Ker a =\left\{ \begin{aligned}
&\infty &&\text{if } b_{0} > 0 , \\
&0 &&\text{if } b_{0} \leq 0,
\end{aligned} \right.
\qquad
{\rm dim\,Ker}\,a^{*} =\left\{ \begin{aligned}
&\infty &&\text{if } b_{0} < 0 , \\
&0 &&\text{if } b_{0} \geq  0 .
\end{aligned} \right.
\end{equation*}
By \eqref{aa}, we deduce
\begin{equation} \label{au31}
\dim \Ker H =\left\{ \begin{aligned}
&\infty &&\text{if } b_{0} \neq 0 , \\
&0 &&\text{if } b_{0} = 0 .
\end{aligned} \right.
\end{equation}
If $b_0 = 0$, it is shown in \cite{bisu97} that $\sigma(H)$ is purely absolutely continuous. If $b_{0} \neq 0$, then the zero eigenvalue is an isolated point of $\sigma ( H )$. This fact was noticed in \cite{dno} without proof, and later was proved in \cite[Example 6]{b}. An explicit bound for the spectral gap adjoining the origin is contained in Proposition \ref{pau3} below which concerns a considerably more general situation.

\pagebreak

\subsection{Almost periodic magnetic fields} \label{ss42}

\subsubsection{Almost periodic fields close to the periodic ones}  \label{421}

Assume that $b \in \WAP ( \R^{2} ; \R )$ and
\begin{equation} \label{3}
\sum_{\lambda \in J_{0} ( b )} \vert b_{\lambda} \vert \vert \lambda \vert^{- 2} < \infty .
\end{equation}
This class of magnetic fields contains the periodic ones and satisfies the assumptions of Theorem \ref{th1}. Similarly to the periodic case, we can choose the solution $\varphi$ of \eqref{1} as $\varphi = \varphi_{0} + \widetilde{\varphi}$ where
\begin{equation*}
\varphi_{0} ( x ) : =  \frac{b_{0} \vert x \vert^{2}}{4} , \qquad \widetilde{\varphi} ( x ) : = -\sum_{\lambda \in J_{0} ( b )} b_{\lambda} \vert \lambda \vert^{- 2} e_{\lambda} ( x ) , \quad x \in \R^{2} .
\end{equation*}
Again, $\widetilde{\varphi}$ is bounded and \eqref{au32} and, hence, \eqref{au31} hold true. As a matter of fact, this class of almost periodic $b$ is contained in a larger class of magnetic fields, considered in the following

\begin{proposition}[{\cite[Proposition 1.2]{r5}}]\sl \label{pau3}
Let $b =  b_{0} + \widetilde{b}$ where $b_{0} \in \R$ and $\widetilde{b}$ is such that there exists a solution $\widetilde{\varphi} \in C_{\rm b}^{2} ( \R^{2} ; \R )$ of the Poisson equation
\begin{equation} \label{1aa}
\Delta \widetilde{\varphi} ( x ) = \widetilde{b} ( x ) ,  \quad x \in \R^{2} .
\end{equation}
Then, \eqref{au31} holds true. If moreover $b_{0} \neq 0$, the zero eigenvalue is isolated in the spectrum of $H$. More precisely,
\begin{equation} \label{au34}
\dist \big( 0 , \sigma ( H ) \setminus \{ 0 \} \big) \geq 2 \vert b_{0} \vert e^{- 2 \osc \widetilde{\varphi}} ,
\end{equation}
where
\begin{equation*}
\osc \widetilde{\varphi} : = \sup_{x \in \R^{2}} \widetilde{\varphi} ( x ) - \inf_{x \in \R^{2}} \widetilde{\varphi} ( x ) .
\end{equation*}
\end{proposition}

\subsubsection{Almost periodic fields distant from the periodic ones}  \label{422}

We suppose now that $b \in \WAP ( \R^{2} ; \R )$ but possibly \eqref{3} does not hold true. This corresponds to the assumptions of Section \ref{s6}. In this case, we can chose the solution $\varphi$ of \eqref{1} as $\varphi = \varphi_{0} + \widetilde{\varphi}$ with
\begin{equation*}
\varphi_{0} ( x ) : =  \frac{b_{0} \vert x \vert^{2}}{4} , \qquad \widetilde{\varphi} ( x ) : = \sum_{\lambda \in J_{0} ( b )} b_{\lambda} \frac{( \lambda \cdot x )^{2}}{\vert \lambda \vert^{2}} \int_{0}^{1} ( 1 - s ) e_{s \lambda} ( x ) \, d s , \quad x \in \R^{2} ,
\end{equation*}
(see \cite{r6}). Then, $\widetilde{\varphi}$ is well defined and  belongs to the class $C^{2}(\R^{2}; \R)$, but generally it is not bounded. However, similarly to \eqref{au30}, it satisfies
\begin{equation*}
\widetilde{\varphi} ( x ) = o ( \vert x \vert^{2} ) , \quad \vert x \vert \to \infty .
\end{equation*}
Hence, if $b_{0} \neq 0$, we have
\begin{equation} \label{au33}
\varphi ( x ) = \frac{b_{0} \vert x \vert^{2}}{4} ( 1 + o ( 1 ) ) , \quad \vert x \vert \to \infty .
\end{equation}
 Now, \eqref{b}--\eqref{c} and \eqref{au33} easily imply the following

\begin{proposition}\sl \label{pau4}
Let $b \in \WAP( \R^{2} ; \R )$ with $b_{0} \neq 0$. Then $\dim \Ker H = \infty$.
\end{proposition}

\begin{xrem}
The result of Proposition \ref{pau4} follows also from \cite[Theorem 3.11]{rs}. If, under the hypotheses of this proposition, there exists no bounded solution of  \eqref{1aa}, then estimate \eqref{au34} is not applicable. However, \cite[Theorem 3.11]{rs} implies that in this case still there is a gap in $\sigma(H)$ adjoining the origin.
\end{xrem}

We are going to show now that if $b \in \WAP( \R^{2} ; \R )$ and  $b_{0} = 0$, but \eqref{3} doesn't hold true, then the situation is quite different with respect to the periodic case. More precisely, for any given $m \in \N \cup \{ \infty \}$, we will construct almost periodic magnetic fields with vanishing mean value such that $\dim \Ker H = m$. Let
\begin{equation*}
C > 0, \quad K \in \N, \quad \gamma_{k} \in \SSS^{1} , \quad k = 1 , \ldots , K ,
\end{equation*}
with $\gamma_{k} \neq \gamma_{\ell}$ if $k \neq \ell$. Further, let
\begin{equation*}
s > 1 , \quad t > 0 , \quad s - 2 t \leq 1 .
\end{equation*}
We will consider magnetic fields of the form
\begin{equation} \label{6}
b ( x ) = C \sum_{k = 1}^{K} \sum_{n = 1}^{\infty} n^{- s} \cos ( n^{- t} \gamma_{k} \cdot x ) , \quad x \in \R^{2} .
\end{equation}
Then, $b \in \WAP ( \R^{2} ; \R )$ but it doesn't satisfy \eqref{3}. Moreover, a simple calculation yields
\begin{equation} \label{b1}
\varphi ( x ) = \widetilde{\varphi} ( x ) = 2 C \sum_{k = 1}^{K} g_{s , t}(\vert \gamma_{k} \cdot x \vert / 2) , \quad x \in \R^{2} ,
\end{equation}
where
\begin{equation} \label{au70}
g_{s , t} ( r ) : = \sum_{n = 1}^{\infty} n^{- s + 2 t} \sin^{2} ( n^{- t} r ) , \quad r \geq 0 .
\end{equation}
Evidently, $g_{s , t}$ satisfies the estimate
\begin{equation*}
0 \leq g_{s , t} ( r ) \leq \zeta ( s ) r^{2} , \quad r \geq 0 ,
\end{equation*}
where $\zeta$ is the Riemann zeta function, and $g_{s , t} ( r ) = 0$ if and only if $r = 0$. If $t > 0$ and $s = 1 + 2 t$, we will write $g_{t}$ instead of $g_{1 + 2 t , t}$. Let us discuss the asymptotic behavior of $g_{s , t} ( r )$ as $r \to \infty$.

\begin{proposition}\sl \label{pr1}
If $s > 1$ and $s - 2 t < 1$, then
\begin{equation*}
g_{s , t} ( r ) = C_{s , t} r^{\frac{- s + 2 t + 1}{t}} ( 1 + o ( 1 ) ) , \quad r \to \infty ,
\end{equation*}
where
$C_{s , t} : = \frac{1}{t} \int_{0}^{\infty} u^{\frac{s- 3t -1}{t}} \sin^{2} ( u ) \, d u$.
\end{proposition}

\begin{xrem}
Under the hypotheses of Proposition \ref{pr1} we have $0 < \frac{- s + 2 t + 1}{t} < 2$.
\end{xrem}

\noindent
We prove Proposition \ref{pr1} in Section \ref{sa1} of the Appendix. Before we turn to the asymptotics as $r \to \infty$ of $g_{t} ( r )$ in the border-line case $t>0$ and $s = 1 + 2 t$, we state an elementary global estimate of $g_{t}$ in this case:

\begin{proposition}\sl \label{pau5}
If $s > 1$ and $s - 2 t = 1$, then
\begin{equation*}
g_{t} ( r ) \leq \ln \big( 1 + r^{1 / t} \big) + C_{t} , \quad r \geq 0 ,
\end{equation*}
where $C_{t} : = \zeta ( 1 + 2 t) + \sup_{n \in \N} ( \sum_{k = 1}^{n} k^{- 1} - \ln n )$.
\end{proposition}

\noindent
The simple proof of Proposition \ref{pau5} can be found in Section \ref{sa2} of the Appendix.

\begin{proposition}\sl \label{pr2}
If $s > 1$ and $s - 2 t = 1$, then
\begin{equation*}
g_{t}(r) = \frac{1}{2t} \ln r ( 1 + o ( 1 ) ) , \quad r \to \infty .
\end{equation*}
\end{proposition}

\noindent
The proof of Proposition \ref{pr2} which is considerably more complicated than that of Proposition \ref{pr1}, is contained in Section \ref{sa3} of the Appendix.

\begin{xrem}
Note that $g_{s , t}$ is represented by  {\em a Dirichlet series}, and extends  to an entire function on the complex plane. More precisely, we have
\begin{equation*}
g_{s , t} ( z ) = 2 z^{2} \sum_{n = 0}^{\infty} \frac{( - 4 )^{n} \zeta ( s + 2 n t )}{(2 ( n + 1 ) ) !} z^{2 n} , \quad z \in \C .
\end{equation*}
The asymptotic behavior at infinity of entire functions whose coefficients involve values of the Riemann zeta function are of a considerable interest in analytic number theory (see e.g.
\cite[Section 14.32]{t}, \cite{ka}, \cite{lkt}). Since we didn't find in the literature the results of Propositions \ref{pr1} and \ref{pr2} which we needed, we include their detailed proofs with the hope that they could be useful to number theorists.
\end{xrem}

\begin{thm}\sl \label{th3}
Suppose that $b$ has the form \eqref{6} with
\begin{equation*}
s > 1 , \quad s - 2 t < 1 , \quad C = 1 , \quad K = 2 , \quad \gamma_{1} = ( 1 , 0 ) ,  \quad \gamma_{2} = ( 0 , 1 ) .
\end{equation*}
Then,
\begin{equation} \label{au60}
\dim \Ker H = \infty .
\end{equation}
\end{thm}

\begin{proof}
By Proposition \ref{pr1}, we have $\varphi ( x ) \asymp \vert x \vert^{\frac{- s + 2 t + 1 }{t} }$ for large $\vert x \vert$. Therefore, $z^{m} e^{- \varphi} \in L^{2} ( \R^{2} )$ for any $m \in \Z_{+}$. By \eqref{b}, we conclude that $\dim \Ker a = \infty$ which combined with \eqref{aa} implies \eqref{au60}.
\end{proof}

\begin{xrem}
The operator $H$ considered in Theorem \ref{th3} falls under Section \ref{s6} but, in general, not under Section \ref{s5}. We don't know yet in general whether under the hypotheses of Theorem \ref{th3} the zero eigenvalue of $H$ is isolated in $\sigma(H)$. Nevertheless, the following proposition shows that there is no spectral gap if $s - t > 1$.
\end{xrem}

\begin{proposition}\sl \label{g1}
Assume that \eqref{1} has a solution $\varphi \in C^{2} ( \R^{2} ; \R )$ such that $\varphi ( x )$ $\geq 0$ for large $\vert x \vert$ and $\varphi ( x ) = o ( |x| )$ as $\vert x \vert \to \infty$. Then, $0$ is a limit point of $\sigma ( H^{+} )$.
\end{proposition}

\begin{proof}
For $\varepsilon > 0$, set $u ( x  ) = e^{\varphi ( x ) - \varepsilon \< x \>}$, $x \in \rd$. We have $u \in \Dom a^*$ and
\begin{equation*}
\< H^{+} u , u \> = \Vert a^{*} u \Vert^{2} = 4 \int e^{2 \varphi} \big\vert \partial_{z} e^{- \varepsilon \< x \>} \big\vert^{2} d x = 4 \varepsilon^{2} \int \vert \partial_{z} \< x \> \vert^{2} \vert u ( x ) \vert^{2} d x \leq C \varepsilon^{2} \Vert u \Vert^{2} ,
\end{equation*}
with $C = 4 \Vert \partial_{z} \< x \> \Vert^{2}_{L^{\infty}} < \infty$. Hence, $\inf \sigma ( H^{+} ) \leq C \varepsilon^{2}$ for any $\varepsilon > 0$. By $H^{+} \geq 0$, this implies $\inf \sigma ( H^{+} ) = 0$. Since $\varphi ( x ) \geq 0$ for large $\vert x \vert$, the kernel of $H^{+}$ is trivial by \eqref{c}, and $0$ has to be a limit point in $\sigma ( H^{+} )$.
\end{proof}

\begin{thm}\sl \label{th4}
Assume that $b$ has the form \eqref{6} with $t > 0$, $s = 1 + 2 t$,
\begin{equation} \label{au68}
C = \frac{1}{K} , \quad \gamma_{k} = ( \cos \theta_{k} , \sin \theta_{k} ) , \quad \theta_{k} = \frac{2 \pi k}{K}, \quad k = 1 , \ldots , K .
\end{equation}
Moreover, suppose that $t^{- 1} \not \in \N$, $K \geq 3$ is odd, and
\begin{equation} \label{au63}
\lfloor t^{- 1} \rfloor < \frac{K - 1}{K t} <  \frac{K + 1}{K t} < \lfloor t^{- 1} \rfloor + 1 .
\end{equation}
Then,
\begin{equation} \label{4}
\dim \Ker H = \lfloor t^{- 1} \rfloor .
\end{equation}
\end{thm}

\begin{xrem}
(i) Assume the hypotheses of Theorem \ref{th4}. In particular, \eqref{3a} holds true and the set $J ( b )$ is bounded. Suppose moreover that $t < 1$. Then, \eqref{4} implies that the zero eigenvalue of $H$ is of finite multiplicity. This result doesn't contradict Theorem \ref{th1} (iii) which holds true almost surely. On the other hand, by Corollary \ref{f1a}, $H$ cannot have isolated eigenvalues of finite multiplicity which, in this particular case, follows also from \cite[Theorem 10.1]{shu1}. Hence, we have $0 \in \sigma_{\rm ess} ( H )$. Of course, this last fact is also implied  by Proposition \ref{g1}.

(ii) Theorem \ref{th4} is valid under much more general hypotheses concerning the family $\{\gamma_{k} \}_{k = 1}^{K} \subset \SSS^{1}$, In particular, if $\gamma_{k}$ are defined as in \eqref{au68}, we can assume that $K \geq 4$ is even; in this case the numbers $K \pm 1$ in \eqref{au63} have to be replaced by $K \pm 2$. Since, anyway, Theorem \ref{th4} should be regarded rather as a pioneering example of 2D Pauli operators with almost periodic fields which admit eigenvalues of finite multiplicity than an exhaustive description of such operators, we decided not to treat more general families $\{ \gamma_{k} \}$, but to make our construction as explicit and simple as possible.
\end{xrem}

\begin{proof}[Proof of Theorem \ref{th4}]
First, we will prove that
\begin{equation} \label{au50}
\dim \Ker a = \lfloor t^{- 1} \rfloor .
\end{equation}
By \eqref{b}, $u \in \Ker a$ implies that $u = e^{- \varphi} f$ with entire $f$. Let us show that under our hypotheses, $f$ is a holomorphic polynomial. Since $\re f$ and $\im f$ are harmonic, it follows from \cite[Section 2.2, Theorem 7]{evans}) that for any $m \in \Z_+$ there exists a constant $c_{m}$ such that for any $z \in \C$ and $R \in ( 0 , \infty )$ we have
\begin{equation*}
\vert f^{( m )} ( z ) \vert \leq \frac{c_{m}}{R^{2 + m}} \int_{B_{R} ( z )} \vert f ( x_{1} + i x_{2} ) \vert \, d x ,
\end{equation*}
where $B_{R} ( z ) = \{ w \in \C  \, | \, \ \vert z - w \vert < R \}$. Combining this estimate with the Cauchy--Schwarz inequality, we get
\begin{equation} \label{au39}
\vert f^{( m )} ( z ) \vert \leq \frac{c_{m}}{R^{2 + m}} \bigg( \int_{B_{R} ( z )} e^{2 \varphi} d x \bigg)^{1 / 2} \Vert u \Vert_{L^{2} ( \R^{2} )} .
\end{equation}
By Proposition \ref{pau5}, we have
\begin{equation} \label{au40}
\int_{B_{R} ( z )} e^{2 \varphi} d x \leq \int_{B_{R + \vert z \vert} ( 0 )} e^{2 \varphi} d x \leq 2 \pi e^{2 C_{t}} \int_{0}^{R + \vert z \vert} \big( 1 + ( r / 2)^{t^{- 1}} \big)^{4} r \, d r = \CO ( R^{4 t^{- 1} + 2} ) ,
\end{equation}
for large $R$. Letting  $R \to \infty$, we find that \eqref{au39} and \eqref{au40} imply that $f^{( m )} ( z ) = 0$ if $2 + m > 2 t^{- 1} + 1$. Since $z$ is arbitrary, $f$ is a polynomial.

Let us now calculate the maximal possible degree of $f$. To this end, we will need a suitable partition of the unit circle $\SSS^{1}$. Since $K \geq 3$ is odd, it is easy to see that there exist disjoint open arcs $\alpha_{\ell} \subset \SSS^{1}$, $\ell \in \{ 1 , \ldots , 2 K \}$, such that $\SSS^{1} = \bigcup_{\ell = 1}^{2 K} \overline{\alpha_{\ell}}$ and for each $\ell = 1 , \ldots , 2 K$ there exists a unique $k_{\ell} \in \{ 1 , \ldots , K \}$ such that
\begin{equation} \label{au45}
\min_{k \neq k_{\ell}} \inf_{\nu \in \alpha_{\ell}} \vert \gamma_{k} \cdot \nu \vert > 0 .
\end{equation}
Next, pick $\varepsilon \in (0,1)$ such that
\begin{equation} \label{au41}
\lfloor t^{-1}\rfloor < (1-\veps)\frac{K-1}{Kt} < (1+\veps) \frac{K+1}{Kt} < \lfloor t^{-1}\rfloor + 1,
\end{equation}
which is possible thanks to \eqref{au63}. By Proposition \ref{pr2} and \eqref{au45}, there exists $C = C_{\varepsilon , t , K} > 0$ such that
\begin{equation} \label{au47}
\frac{( 1 - \varepsilon )}{2 t} \ln \< x \> - C \leq g_{t} ( \vert \gamma_{k} \cdot x \vert / 2 ) \leq \frac{( 1 + \varepsilon )}{2 t} \ln \< x \> + C , \quad k \neq k_{\ell} , \quad \frac{x}{\vert x \vert} \in \alpha_{\ell} ,
\end{equation}
for all $\ell = 1 , \ldots , 2 K$. Assume that $t < 1$ i.e. $\lfloor t^{- 1}\rfloor \geq 1$, and pick $m \in \Z_{+}$ with $m \leq \lfloor t^{- 1} \rfloor - 1$. We will show that $u = z^{m} e^{- \varphi} \in L^{2} ( \R^{2} )$. Using the first inequality of \eqref{au47} to estimate the contribution of $k \neq k_{\ell}$ and $g_{t} ( y ) \geq 0$ for any $y \in \R^{2}$ to estimate the contribution of $k = k_{\ell}$, we find that the function $\varphi$ given in \eqref{b1} satisfies
\begin{equation*}
\varphi ( x ) \geq ( 1 - \varepsilon ) \frac{K - 1}{K t} \ln \< x \> - C , \quad x \in \R^{2} ,
\end{equation*}
with an appropriate constant $C$. Therefore,
\begin{equation} \label{au44}
\Vert u \Vert_{L^{2} ( \R^{2} )}^{2} = \int_{\R^{2}} \vert x \vert^{2 m} e^{- 2 \varphi ( x )} d x \leq e^{2 C} \int_{\R^{2}} \vert x \vert^{2 m} \< x \>^{- 2 ( 1 - \varepsilon ) \frac{K - 1}{K t}} d x .
\end{equation}
Since $m \leq \lfloor t^{- 1} \rfloor - 1$, the first inequality in \eqref{au41} yields
\begin{equation*}
2 m - 2 ( 1 - \varepsilon ) \frac{K - 1}{K t} < - 2 ,
\end{equation*}
and the last integral in \eqref{au44} is convergent. Thus we find that if $t < 1$, then
\begin{equation} \label{au49}
\dim \Ker a \geq \lfloor t^{- 1} \rfloor .
\end{equation}
If $t \geq 1$, i.e. $\lfloor t^{- 1} \rfloor = 0$, then \eqref{au49} is trivially true.

Assume now $t > 0$ and pick $m \in \Z_{+}$ with $m \geq \lfloor t^{- 1} \rfloor$. We will show that the function $u = z^{m} e^{- \varphi} \not \in L^{2} ( \R^{2} )$. Using the second inequality of \eqref{au47} to estimate the contribution of $k \neq k_{\ell}$ and Proposition \ref{pau5} to estimate the contribution of $k = k_{\ell}$, we find that
\begin{equation*}
\varphi ( x ) \leq ( 1 + \varepsilon ) \frac{K + 1}{K t} \ln \< x \> + C , \quad x \in \R^{2} ,
\end{equation*}
with an appropriate constant $C$. Thus,
\begin{equation} \label{b2}
\Vert u \Vert_{L^{2} ( \R^{2} )}^{2} = \int_{\R^{2}} \vert x \vert^{2 m} e^{- 2 \varphi ( x )} d x \geq e^{- 2 C} \int_{\R^{2}} \vert x \vert^{2 m} \< x \>^{- 2 ( 1 + \varepsilon ) \frac{K + 1}{K t}} d x .
\end{equation}
Since $m \geq \lfloor t^{- 1} \rfloor$, the second inequality in \eqref{au41} gives
\begin{equation*}
2 m - 2 ( 1 + \varepsilon ) \frac{K + 1}{K t} > - 2 ,
\end{equation*}
and the last integral in \eqref{b2} is divergent. Then $\dim \Ker a < \lfloor t^{- 1} \rfloor + 1$ which combined with \eqref{au49} implies \eqref{au50}. Let us now prove that
\begin{equation} \label{au51}
\dim \Ker a^{*} = 0 .
\end{equation}
By \eqref{c}, $u \in \Ker a^{*}$ is equivalent to $u = e^{\varphi} f$ with entire $\overline{f}$. Under our hypotheses, $e^{\varphi} \geq 1$ so that  $u \in L^{2} ( \R^{2} )$ implies $f \in L^{2} ( \R^{2} )$.  Hence, $f$ vanishes identically which implies \eqref{au51}.
Putting together \eqref{au50} and \eqref{au51}, and taking into account \eqref{aa}, we obtain \eqref{4}.
\end{proof}

\appendix

\section{Asymptotics of Dirichlet series} \label{sa}

\subsection{Proof of Proposition \ref{pr1}} \label{sa1}
First, we estimate the contribution of the large $n$'s in \eqref{au70}. Consider $M > 1$. Using $\vert \sin y \vert \leq \vert y \vert$ and $- s < - 1$, we deduce
\begin{align}
\sum_{n \geq M r^{1 / t}} n^{- s + 2 t} \sin^{2} \Big( \frac{r}{n^{t}} \Big) &\leq \sum_{n \geq M r^{1 / t}} n^{- s + 2 t} r^{2} n^{- 2 t}  \nonumber \\
&\leq C r^{2} ( M r^{1 / t} )^{1 - s} = C M^{1 - s} r^{\frac{1 - s + 2 t}{t}} . \label{a1}
\end{align}
Here and in the sequel, $C$ will denote a positive constant which may only depend on $s$ and $t$.

We now deal with the contribution of the small $n$'s. Consider $\varepsilon > 0$. Using that $\vert \sin y \vert \leq 1$ and $- s + 2 t > - 1$, we obtain
\begin{equation} \label{a2}
\sum_{n \leq \varepsilon r^{1 / t}} n^{- s + 2 t} \sin^{2} \Big( \frac{r}{n^{t}} \Big) \leq \sum_{n \leq \varepsilon r^{1 / t}} n^{- s + 2 t} \leq C ( \varepsilon r^{1 / t} )^{1 - s + 2 t} = C \varepsilon^{1 - s + 2 t} r^{\frac{1 - s + 2 t}{t}} .
\end{equation}

It remains to deal with the $n$'s of size $r^{1 /t}$. Let $n \in [ \varepsilon r^{1 / t} , M r^{1 / t} ]$. For $y \in [ n , n + 1]$, we have
\begin{equation} \label{a3}
y^{- s + 2 t} = n^{- s + 2 t} + \CO ( n^{- s + 2 t - 1} ) ,
\end{equation}
and
\begin{equation*}
y^{- t} = n^{- t} + \CO ( n^{- t - 1} ) .
\end{equation*}
In particular,
\begin{equation} \label{a4}
\Big\vert \sin^{2} \Big( \frac{r}{y^{t}} \Big) - \sin^{2} \Big( \frac{r}{n^{t}} \Big) \Big\vert \leq C r n^{- t - 1} \leq R r^{- 1 / t} ,
\end{equation}
where $R$ is a constant which may depend on $s , t , \varepsilon , M$. Inequalities \eqref{a3} and \eqref{a4} imply
\begin{align}
\bigg\vert n^{- s + 2 t} &\sin^{2} \Big( \frac{r}{n^{t}} \Big) - \int_{n}^{n+1} y^{- s + 2 t} \sin^{2} \Big( \frac{r}{y^{t}} \Big) \, d y \bigg\vert \nonumber \\
&= \bigg\vert \int_{n}^{n+1} \Big( n^{- s + 2 t} \sin^{2} \Big( \frac{r}{n^{t}} \Big) -  y^{- s + 2 t} \sin^{2} \Big( \frac{r}{y^{t}} \Big) \Big) \, d y \bigg\vert  \nonumber \\
&\leq n^{- s + 2 t} \int_{n}^{n + 1} \Big\vert \sin^{2} \Big( \frac{r}{n^{t}} \Big) - \sin^{2} \Big( \frac{r}{y^{t}} \Big) \Big\vert \, d y + \int_ {n}^{n + 1} \big\vert n^{- s + 2 t} - y^{- s + 2 t} \big\vert \, d y   \nonumber \\
&\leq R r^{\frac{-s + 2 t - 1}{t}} .
\end{align}
Summing this estimate over $n$ gives
\begin{align}
\bigg\vert \sum_{\varepsilon r^{1 / t} \leq n \leq M r^{1 / t}} n^{- s + 2 t} \sin^{2} &\Big( \frac{r}{n^{t}} \Big) - \int_{\varepsilon r^{1 / t}}^{M r^{1 / t}} y^{- s + 2 t} \sin^{2} \Big( \frac{r}{y^{t}} \Big) \, d y \bigg\vert \nonumber \\
&\leq R r^{\frac{-s + 2 t - 1}{t}} ( M - \varepsilon ) r^{1 / t} + R r^{\frac{- s + 2 t}{t}} \leq R r^{\frac{- s + 2 t}{t}} .  \label{a5}
\end{align}

Changing the variable $y = r^{1 / t} u^{- 1 / t}$, we find that
\begin{align}
\int_{\varepsilon r^{1 / t}}^{M r^{1 / t}} y^{- s + 2 t} \sin^{2} \Big( \frac{r}{y^{t}} \Big) \, d y
&= \frac{r^{\frac{1 - s + 2 t }{t}}}{t} \int_{M^{- t}}^{\varepsilon^{- t}} u^{\frac{s - 3 t - 1}{t}} \sin^{2} ( u ) \, d u . \label{a6}
\end{align}
Now the integrand behaves like
\begin{equation*}
u^{\frac{s - 3 t - 1}{t}} \sin^{2} ( u ) \leq \left\{ \begin{aligned}
&u^{\frac{s - t - 1}{t}} = u^{- 1 + \frac{s - 1}{t}} &&\text{ near } 0 ,  \\
&u^{\frac{s - 3 t - 1}{t}} = u^{- 1 + \frac{s - 2 t - 1}{t}} &&\text{ near } + \infty .
\end{aligned} \right.
\end{equation*}
Since $s - 1 > 0$ and $s - 2 t - 1 < 0$, this function is integrable on $(0,\infty)$, and \eqref{a6} becomes
\begin{align*}
\int_{\varepsilon r^{1 / t}}^{M r^{1 / t}} & y^{- s + 2 t} \sin^{2} \Big( \frac{r}{y^{t}} \Big) \, d y \\
&\quad = \frac{r^{\frac{1 - s + 2 t }{t}}}{t} \Big( \int_{0}^{+ \infty} u^{\frac{s - 3 t - 1}{t}} \sin^{2} ( u ) \, d u + o_{\varepsilon \to 0} ( 1 ) + o_{M \to + \infty} ( 1 ) \Big) .
\end{align*}

Eventually, combining the previous estimate together with \eqref{a1}, \eqref{a2} and \eqref{a4}, we deduce
\begin{align}
\Big\vert g_{s , t} ( r ) - & \frac{r^{\frac{1 - s + 2 t }{t}}}{t} \int_{0}^{+ \infty} u^{\frac{s - 3 t - 1}{t}} \sin^{2} ( u ) \, d u \Big\vert \leq C M^{1 - s} r^{\frac{1 - s + 2 t}{t}}   \nonumber \\
&+ C \varepsilon^{1 - s + 2 t} r^{\frac{1 - s + 2 t}{t}} + R r^{\frac{- s + 2 t}{t}} + r^{\frac{1 - s + 2 t }{t}} \big( o_{\varepsilon \to 0} ( 1 ) + o_{M \to + \infty} ( 1 ) \big) .
\end{align}
Then, we obtain the proposition taking $\varepsilon$ small enough and $M$ large enough. Here, we have used again that $1 - s < 0$ and $1 - s + 2 t > 0$.

\subsection{Proof of Proposition \ref{pau5}} \label{sa2}

Let $r \geq 0$ and $\rho : = \lfloor 1 + r^{1/t} \rfloor $. In particular, $\rho \geq 1$. Then we can write any $n \in \N$ as $n = p \rho + q$ with $p \in \Z_{+}$ and $q = 1,\ldots,\rho$, and this representation is unique. Therefore, since $\rho \geq r^{1/t}$ and $\sin^{2} ( \alpha ) \leq \min\{1,\alpha^{2}\}$ for $\alpha \in \R$, we find that
\begin{align*}
g_{t} ( r ) & = \sum_{p = 0}^{\infty} \sum_{q = 1}^{\rho} ( p \rho + q )^{- 1} \sin^{2} \bigg( \bigg( \frac{r^{1 / t}}{p \rho + q} \bigg)^{t} \bigg)  \\
& = \sum_{q = 1}^{\rho} q^{- 1} \sin^{2} ( r q^{-t} )  + \sum_{p=1}^{\infty} \sum_{q = 1}^{\rho} ( p \rho + q )^{- 1} \sin^{2} \bigg( \bigg( \frac{r^{1 / t}}{p \rho + q} \bigg)^{t} \bigg)  \\
& \leq \sum_{q = 1}^{\rho} q^{- 1} + \sum_{p = 1}^{\infty} \rho ( p \rho + 1 )^{- 1} p^{- 2 t}  \\
& \leq \ln \rho +  \sum_{q = 1}^{\rho} q^{- 1} - \ln \rho + \zeta ( 1 + 2 t ) \\
& \leq \ln (1 + r^{1 / t} ) + C_{t} ,
\end{align*}
which implies Proposition \ref{pau5}.

\subsection{Proof of Proposition \ref{pr2}}  \label{sa3}

Considering a given $\varepsilon > 0$, we will show that $g_{t} ( r )$ is at distance at most $\varepsilon \ln r$ from $\ln r / 2 t$ for $r$ large enough. For that, we decompose the sum over $\N$ into different zones which are summarized in Figure \ref{f1}.

\begin{figure}
\begin{center}
\begin{picture}(0,0)%
\includegraphics{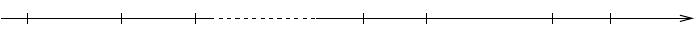}%
\end{picture}%
\setlength{\unitlength}{1105sp}%
\begingroup\makeatletter\ifx\SetFigFont\undefined%
\gdef\SetFigFont#1#2#3#4#5{%
  \reset@font\fontsize{#1}{#2pt}%
  \fontfamily{#3}\fontseries{#4}\fontshape{#5}%
  \selectfont}%
\fi\endgroup%
\begin{picture}(19866,1486)(-11432,-125)
\put(-10649,-61){\makebox(0,0)[b]{\smash{{\SetFigFont{9}{10.8}{\rmdefault}{\mddefault}{\updefault}$1$}}}}
\put(6001,-61){\makebox(0,0)[b]{\smash{{\SetFigFont{9}{10.8}{\rmdefault}{\mddefault}{\updefault}$r^{\frac{1}{t} + \delta}$}}}}
\put(-149,1214){\makebox(0,0)[b]{\smash{{\SetFigFont{9}{10.8}{\rmdefault}{\mddefault}{\updefault}$J_{1}$}}}}
\put(2551,1214){\makebox(0,0)[b]{\smash{{\SetFigFont{9}{10.8}{\rmdefault}{\mddefault}{\updefault}$I_{0}$}}}}
\put(-6899,1214){\makebox(0,0)[b]{\smash{{\SetFigFont{9}{10.8}{\rmdefault}{\mddefault}{\updefault}$J_{K}$}}}}
\put(-7949,-61){\makebox(0,0)[b]{\smash{{\SetFigFont{9}{10.8}{\rmdefault}{\mddefault}{\updefault}$r^{\frac{1}{t + K} - \delta}$}}}}
\put(-5849,-61){\makebox(0,0)[b]{\smash{{\SetFigFont{9}{10.8}{\rmdefault}{\mddefault}{\updefault}$r^{\frac{1}{t + K} + \delta}$}}}}
\put(-1049,-61){\makebox(0,0)[b]{\smash{{\SetFigFont{9}{10.8}{\rmdefault}{\mddefault}{\updefault}$r^{\frac{1}{t + 1} - \delta}$}}}}
\put(751,-61){\makebox(0,0)[b]{\smash{{\SetFigFont{9}{10.8}{\rmdefault}{\mddefault}{\updefault}$r^{\frac{1}{t + 1} + \delta}$}}}}
\put(4351,-61){\makebox(0,0)[b]{\smash{{\SetFigFont{9}{10.8}{\rmdefault}{\mddefault}{\updefault}$r^{\frac{1}{t} - \delta}$}}}}
\put(5176,1214){\makebox(0,0)[b]{\smash{{\SetFigFont{9}{10.8}{\rmdefault}{\mddefault}{\updefault}$J_{0}$}}}}
\put(8101,1064){\makebox(0,0)[b]{\smash{{\SetFigFont{9}{10.8}{\rmdefault}{\mddefault}{\updefault}$n \in \N$}}}}
\end{picture}%
\end{center}
\caption{The different regions considered in the proof of Proposition \ref{pr2}.} \label{f1}
\end{figure}
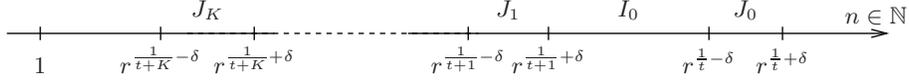

\underline{Step 1:} treatment of the large values of $n$. We set
\begin{equation} \label{a9}
\delta = \min \Big( \frac{1}{2 t + 2 \varepsilon^{- 1} + 2} , \frac{\varepsilon}{2 \varepsilon^{- 1} + 4} \Big).
\end{equation}
Then $0 < \delta \leq \veps$. Let
\begin{equation*}
u_{n} ( r ) : = \frac{1}{n} \sin^{2} \Big( \frac{r}{n^{t}} \Big)
\end{equation*}
be the generic term of the series which defines $g_{t} ( r )$. Then, $\vert \sin y \vert \leq \vert y \vert$ yields
\begin{equation} \label{a10}
\sum_{n \geq r^{\frac{1}{t} + \delta}} \vert u_{n} \vert \leq \sum_{n \geq r^{\frac{1}{t} + \delta}} r^{2} n^{- 1 - 2 t} \lesssim r^{2} r^{- 2 t ( \frac{1}{t} + \delta )} = r^{- 2 t \delta} \leq \varepsilon \ln r ,
\end{equation}
for $r$ large enough.

\underline{Step 2:} decomposition of $u_{n}$. We write
\begin{equation} \label{a35}
u_{n} ( r ) = \frac{1}{2 n} - \frac{v_{n} ( r ) + i \overline{v_{n} ( r )}}{2} \qquad \text{with} \qquad v_{n} ( r ) =\frac{1}{n} e^{2 i r n^{- t}} .
\end{equation}
From \eqref{a9} and the asymptotics of the harmonic series, we deduce
\begin{equation} \label{a11}
\sum_{n < r^{\frac{1}{t} + \delta}} \frac{1}{2 n} = \frac{1}{2} \ln r^{\frac{1}{t} + \delta} + \CO ( 1 ) = \frac{\ln r}{2 t} + \varrho \qquad \text{with} \qquad \vert \varrho \vert \leq \varepsilon \ln r ,
\end{equation}
for $r$ large enough. It remains to study the sum of $v_{n}$ for $n < r^{\frac{1}{t} + \delta}$.

\underline{Step 3:} treatment of the small values of $n$. Let
\begin{equation*}
K = \lceil \varepsilon^{- 1} \rceil \in \N ,
\end{equation*}
where $\lceil y \rceil$ denotes the smallest integer greater or equal to $y$. Note that \eqref{a9} guarantees that $\delta < ( t + K )^{- 1}$. Using again the asymptotics of the harmonic series, we get
\begin{equation} \label{a34}
\sum_{n \leq r^{\frac{1}{t + K} - \delta}} \vert v_{n} \vert = \sum_{n \leq r^{\frac{1}{t + K} - \delta}} \frac{1}{n} = \ln r^{\frac{1}{t + K} - \delta} + \CO ( 1 ) \leq \frac{\ln r}{t + K} \leq \varepsilon \ln r ,
\end{equation}
for $r$ large enough. It remains to study the contribution of
\begin{equation*}
\big( r^{\frac{1}{t + K} - \delta} , r^{\frac{1}{t} + \delta} \big) = \Big( \bigcup_{k = 0}^{K - 1} I_{k} \Big) \cup \Big( \bigcup_{k = 0}^{K} J_{k} \Big) ,
\end{equation*}
with
\begin{equation} \label{a23}
I_{k} = \big[ r^{\frac{1}{t + k + 1} + \delta} , r^{\frac{1}{t + k} - \delta} \big] \qquad \text{ and } \qquad J_{k} = \big( r^{\frac{1}{t+k} - \delta} , r^{\frac{1}{t+k} + \delta} \big) .
\end{equation}
Note that by \eqref{a9}, the intervals $I_{k}$ and $J_{k}$ are non-empty.

\underline{Step 4:} contribution of the small regions $J_{k}$ for $k = 0 , \ldots , K$. Using \eqref{a9} and
\begin{equation*}
\sum_{a \leq n \leq b} \frac{1}{n} \leq \ln b - \ln a + 1,
\end{equation*}
for all $1 \leq a < b$, we deduce
\begin{equation} \label{a33}
\sum_{n \in J_{k} \cap \N} \vert v_{n} \vert = \sum_{n \in J_{k} \cap \N} \frac{1}{n} \leq \ln r^{\frac{1}{t+k} + \delta} - \ln r^{\frac{1}{t+k} - \delta} + 1 \leq 2 \delta \ln r + 1 \leq \frac{\varepsilon}{K + 1} \ln r ,
\end{equation}
for $r$ large enough.

\underline{Step 5:} the iterated derivatives. It remains to study the contribution of the bands $I_{k}$ for $k = 0 , \ldots , K - 1$. Here, we can not bound directly the sum and must use some cancelations. For that, we first define the derivatives of the phase function $2 r n^{- t}$ with respect to $n$.

For $\{ h_{j} \}_{j \in \N} \subset \N$ and $y \in [ 1 , \infty )$, we define $a_{y}^{( j )} ( h_{1} , \ldots , h_{j} )$ by induction over $j \in \Z_{+}$ by $a_{y}^{( 0 )} = 2 r y^{- t}$ and
\begin{equation} \label{a16}
a_{y}^{( j )} ( h_{1} , \ldots , h_{j} ) = a_{y + h_{j}}^{( j - 1 )} ( h_{1} , \ldots , h_{j - 1} ) - a_{y}^{( j - 1 )} ( h_{1} , \ldots , h_{j - 1} ) .
\end{equation}
Roughly speaking, $a_{y}^{( j )}$ is the $j^{\text{st}}$ discrete derivative of $a_{y}^{( 0 )}$ with respect to $y$. Set
\begin{equation*}
S(w) = \big\{ f \in C^{\infty} ( [ 1 , \infty ) ) \, | \, \ \vert \partial_{y}^{k} f ( y ) \vert \lesssim y^{- k} w ( y ) , \ k \in \Z_{+} , \ y \in [ 1 , \infty ) \big\} ,
\end{equation*}
$w > 0$ being an appropriate weight function. We will write $f_{1} = f_{2} \modu S ( w )$ if $f_{1} - f_{2} \in S ( w )$.

\begin{lem}\sl \label{a12}
For any $j \in \Z_{+}$, there exist symbols $P_{j} \in S ( y^{2^{j} t - t - j - 1} )$ and $Q_{j} \in S ( y^{2^{j} t - 1} )$ such that
\begin{equation*}
a_{y}^{( j )} ( h_{1} , \ldots , h_{j} ) = 2 r ( - 1 )^{j} \prod_{\ell = 1}^{j} ( \ell + t - 1 ) h_{\ell} \frac{y^{2^{j} t - t - j} + P_{j} ( y )}{y^{2^{j} t} + Q_{j} ( y )} .
\end{equation*}
\end{lem}

\begin{xrem}
The functions $P_{j} , Q_{j}$ may depend on $h_{1} , \ldots , h_{j} \in \N$. Note that the prefactor never vanishes since $t > 0$. When $t$ is an integer, $P_{j} , Q_{j}$ are polynomials in $n$.
\end{xrem}

\begin{proof}[Proof of Lemma \ref{a12}]
We show this property by induction over $j \in \Z_{+}$. By definition, it is satisfied for $j = 0$ with $P_{0} = Q_{0} = 0$. We assume that it holds true for some $j \in \Z_{+}$. Then, using the shorthand notation
\begin{equation} \label{a18}
C_{j} = 2 r ( - 1 )^{j} \prod_{\ell = 1}^{j} ( \ell + t - 1 ) h_{\ell} ,
\end{equation}
we can write
\begin{equation} \label{a13}
a_{y}^{( j + 1)} = C_{j} \frac{( y + h_{j + 1} )^{2^{j} t - t - j} + P_{j} ( y + h_{j + 1} )}{( y + h_{j + 1} )^{2^{j} t} + Q_{j} ( y + h_{j + 1} )} - C_{j} \frac{y^{2^{j} t - t - j} + P_{j} ( y )}{y^{2^{j} t} + Q_{j} ( y )} = C_{j} \frac{A ( y )}{B ( y )} ,
\end{equation}
with
\begin{align*}
A ( y ) = \big( ( y + h_{j + 1} &)^{2^{j} t - t - j} + P_{j} ( y + h_{j + 1} ) \big) \big( y^{2^{j} t} + Q_{j} ( y ) \big) \\
&- \big( y^{2^{j} t - t - j} + P_{j} ( y ) \big) \big( ( y + h_{j + 1} )^{2^{j} t} + Q_{j} ( y + h_{j + 1} ) \big) ,
\end{align*}
and
\begin{equation*}
B ( y ) = \big( ( y + h_{j + 1} )^{2^{j} t} + Q_{j} ( y + h_{j + 1} ) \big) \big( y^{2^{j} t} + Q_{j} ( y ) \big) .
\end{equation*}
The Taylor formula implies
\begin{equation*}
( y + h_{j + 1} )^{\alpha} = y^{\alpha} ( 1 + h_{j + 1} / y )^{\alpha} = y^{\alpha} + \alpha h_{j + 1} y^{\alpha - 1} \modu S ( y^{\alpha - 2} ) ,
\end{equation*}
for all $\alpha \in \R$ and
\begin{align*}
P_{j} ( y + h_{j + 1} ) &= P_{j} ( y ) + \int_{0}^{h_{j + 1}} P_{j}^{\prime} ( y + s ) \, d s = P_{j} ( y ) \modu S ( y^{2^{j} t - t - j - 2} ) ,  \\
Q_{j} ( y + h_{j + 1} ) &= Q_{j} ( y ) \modu S ( y^{2^{j} t - 2} ) .
\end{align*}
Combining the previous estimates, $A ( y )$ becomes
\begin{align}
A ( y ) ={}& \big( y^{2^{j} t - t - j} + ( 2^{j} t - t - j ) h_{j + 1} y^{2^{j} t - t - j - 1}     \nonumber \\
&\qquad \qquad \qquad \qquad \quad + P_{j} ( y ) \modu S ( y^{2^{j} t - t - j - 2} ) \big) \big( y^{2^{j} t} + Q_{j} ( y ) \big)  \nonumber  \\
&- \big( y^{2^{j} t - t - j} + P_{j} ( y ) \big) \big( y^{2^{j} t} + 2^{j} t h_{j + 1} y^{2^{j} t - 1} + Q_{j} ( y ) \modu S ( y^{2^{j} t - 2} ) \big)    \nonumber \\
={}& - ( j + t ) h_{j + 1} \big( y^{2^{j + 1} t - t - j - 1} + P_{j + 1} ( y ) \big) ,  \label{a14}
\end{align}
for some $P_{j + 1} \in S ( y^{2^{j + 1} t - t - j - 2} )$. Similarly,
\begin{align}
B ( y ) &= \big( y^{2^{j} t} \modu S ( y^{2^{j} t - 1} ) \big) \big( y^{2^{j} t} \modu S ( y^{2^{j} t - 1} ) \big)    \nonumber \\
&= y^{2^{j + 1} t} + Q_{j + 1} ( y ) ,   \label{a15}
\end{align}
for some $Q_{j + 1} \in S ( y^{2^{j + 1} t - 1} )$. Eventually, \eqref{a13} together with \eqref{a14} and \eqref{a15} imply that the conclusions of the lemma hold true for $j + 1$ and then for all $j \in \Z_{+}$.
\end{proof}

\underline{Step 6:} the iterative Van der Corput argument. We will use a standard technique to prove the uniform distribution of sequences called the Van der Corput inequality. A version of this result is stated in the following lemma whose proof can be found in \cite[(3.2)]{kn} (see also \cite[Chapter 2]{rauzy}).

\begin{lem}\sl \label{a17}
Let $\{ b_{n} \}_{n \in \N}$ be a sequence of real numbers. Then, for all $1 \leq H \leq N$, we have
\begin{equation*}
\bigg\vert \frac{1}{N} \sum_{n = 1}^{N} e^{i b_{n}} \bigg\vert^{2} \leq \frac{2}{H} + \frac{4}{H} \sum_{h = 1}^{H - 1} \bigg\vert \frac{1}{N - h} \sum_{n = 1}^{N - h} e^{i ( b_{n + h} - b_{n} )} \bigg\vert .
\end{equation*}
\end{lem}

Mimicking the notations of \eqref{a16}, we recognize $b_{n}^{( 1 )} ( h ) = b_{n + h} - b_{n}$ in the right hand side of the last equation. Then, if we want to show that
\begin{equation*}
\bigg\vert \frac{1}{N} \sum_{n = 1}^{N} e^{i b_{n}} \bigg\vert \leq \varepsilon_{0} ,
\end{equation*}
for some $\varepsilon_{0} > 0$, it is enough to prove that
\begin{equation*}
\bigg\vert \frac{1}{N - h_{1}} \sum_{n = 1}^{N - h_{1}} e^{i b_{n}^{( 1 )} ( h_{1} )} \bigg\vert \leq \frac{\varepsilon_{0}^{2}}{8} = : \varepsilon_{1} ,
\end{equation*}
for all $1 \leq h_{1} < H_{1} : = \lceil 4 \varepsilon_{0}^{- 2} \rceil$. Iterating this argument, it is enough to prove that, for some $J \in \N$,
\begin{equation*}
\bigg\vert \frac{1}{N - h_{1} - \cdots - h_{J}} \sum_{n = 1}^{N - h_{1} - \cdots - h_{J}} e^{i b_{n}^{( J )} ( h_{1} , \ldots , h_{J} )} \bigg\vert \leq \varepsilon_{J} ,
\end{equation*}
for all $1 \leq h_{j} < H_{j}$ with $1\leq j \leq J$. Here, $\varepsilon_{J} > 0$ and $H_{j} \in \N$ only depend on $\varepsilon_{0}$ (and not on $b_{n}$ or $N$), but we assume that $N > H_{1} + \cdots + H_{J}$.

\underline{Step 7:} contribution coming from the interval $I_{k}$. Let us fix $k \in \{ 0 , \ldots , K - 1 \}$ and consider $n_{0} \in I_{k} \cap \N$. We define
\begin{equation} \label{a26}
d_{n} = a_{n_{0} + n}^{( k )} ( h_{1} , \ldots , h_{k} ) - a_{n_{0}}^{( k )} ( h_{1} , \ldots , h_{k} ) .
\end{equation}
From \eqref{a16}, this quantity is nothing more than $a_{n_{0}}^{( k + 1)} ( h_{1} , \ldots , h_{k} , n )$. Our next lemma contains a useful estimate of $d_n$:

\begin{lem}\sl \label{a19}
For $0 \leq n \ll n_{0}$, we have
\begin{equation*}
d_{n} = D M n r n_{0}^{- t - k - 1} \big( 1 + \CO ( n n_{0}^{- 1} ) \big) ,
\end{equation*}
with
\begin{equation*}
D = 2 ( - 1 )^{k + 1} ( k + t ) \prod_{\ell = 1}^{k} ( \ell + t - 1 ) \neq 0 \qquad \text{ and } \qquad M = \prod_{\ell = 1}^{k} h_{\ell} \in \N .
\end{equation*}
\end{lem}

\begin{proof}
We follow the proof of Lemma \ref{a12} with $y = n_{0}$. Using \eqref{a13}, we can write
\begin{equation} \label{a20}
d_{n} = C_{k} \frac{A ( n )}{B ( n )} ,
\end{equation}
 $C_{k}$ being defined in \eqref{a18},
\begin{align*}
A ( n ) = \big( ( n_{0} + n &)^{2^{k} t - t - k} + P_{k} ( n_{0} + n ) \big) \big( n_{0}^{2^{k} t} + Q_{k} ( n_{0} ) \big) \\
&- \big( n_{0}^{2^{k} t - t - k} + P_{k} ( n_{0} ) \big) \big( ( n_{0} + n )^{2^{k} t} + Q_{k} ( n_{0} + n ) \big) ,
\end{align*}
and
\begin{equation*}
B ( n ) = \big( ( n_{0} + n )^{2^{k} t} + Q_{k} ( n_{0} + n ) \big) \big( n_{0}^{2^{k} t} + Q_{k} ( n_{0} ) \big) .
\end{equation*}
For $\alpha \in \R$, the Taylor formula gives
\begin{equation*}
( n_{0} + n )^{\alpha} = n_{0}^{\alpha} ( 1 + n / n_{0} )^{\alpha} = n_{0}^{\alpha} + \alpha n n_{0}^{\alpha - 1} + \CO ( n^{2} n_{0}^{\alpha - 2} ) ,
\end{equation*}
uniformly for $0 \leq n \leq n_{0} / 2$. Analogously, we have
\begin{equation*}
\left. \begin{aligned}
P_{k} ( n_{0} + n ) &= P_{k} ( n_{0} ) + \int_{0}^{n} P_{k}^{\prime} ( n_{0} + s ) \, d s = P_{k} ( n_{0} ) + \CO ( n n_{0}^{2^{k} t - t - k - 2} ) ,  \\
Q_{k} ( n_{0} + n ) &= Q_{k} ( n_{0} ) + \CO ( n n_{0}^{2^{k} t - 2} ) ,
\end{aligned}
\right.
\end{equation*}
uniformly for $0 \leq n \leq n_{0} / 2$. Summing up, we deduce
\begin{equation}
\begin{aligned}
A ( n ) &= - ( k + t ) n n_{0}^{2^{k + 1} t - t - k - 1} \big( 1 + \CO ( n n_{0}^{- 1} ) \big) ,  \\
B ( n ) &= n_{0}^{2^{k + 1} t} \big( 1 + \CO ( n n_{0}^{- 1} ) \big) .
\end{aligned}
\end{equation}
Finally, \eqref{a20} becomes
\begin{align*}
d_{n} &= - C_{k} ( k + t ) n n_{0}^{- t - k - 1} \frac{1 + \CO ( n n_{0}^{- 1} )}{1 + \CO ( n n_{0}^{- 1} )} \\
&= - C_{k} ( k + t ) n n_{0}^{- t - k - 1} \big( 1 + \CO ( n n_{0}^{- 1} ) \big) ,
\end{align*}
for $n \ll n_{0}$.
\end{proof}

Here, it is important to note that \eqref{a23} yields
\begin{equation} \label{a24}
0 < r n_{0}^{- t - k - 1} \leq r^{- \delta ( t + k + 1)} \ll 1 ,
\end{equation}
for $n_{0} \in I_{k}$. Thus, the sequence $\{ d_{n} \}_{n}$ is slowly increasing for $n \ll n_{0}$. Let us define
\begin{equation} \label{a22}
N ( n_{0} ) : = \Big\lceil \frac{2 \pi}{D} r^{- 1} n_{0}^{t + k + 1} \Big\rceil ,
\end{equation}
which is roughly speaking the primitive period of $d_{n} / M$ modulo $2 \pi$. In other words, $N ( n_{0} )$ is such that $d_{N ( n_{0} )} \approx 2 \pi M \in 2 \pi \Z$. From \eqref{a23}, this period satisfies
\begin{equation} \label{a21}
1 \ll \frac{2 \pi}{D} r^{\delta ( t + k + 1)} \leq N ( n_{0} ) \leq \frac{2 \pi}{D} r^{- \delta ( t + k )} n_{0} + 1 \ll n_{0} ,
\end{equation}
for $r$ large enough. Summing $e^{i d_{n}}$ over a period leads to

\begin{lem}\sl \label{a27}
Let $h_{1} , \ldots , h_{k} \in \N$ be fixed. For $r$ large enough, we have
\begin{equation*}
\bigg\vert \sum_{n = 0}^{N ( n_{0} ) - 1} e^{i d_{n}} \bigg\vert \lesssim N ( n_{0} ) r^{- \delta ( t + k )} + 1 ,
\end{equation*}
uniformly for $n_{0} \in I_{k}$.
\end{lem}

\begin{proof}
Lemma \ref{a19} implies
\begin{align}
\sum_{n = 0}^{N ( n_{0} ) - 1} e^{i d_{n}} &= \sum_{n = 0}^{N ( n_{0} ) - 1} \big( e^{i D M n r n_{0}^{- t - k - 1}} + \CO ( n^{2} r n_{0}^{- t - k - 2} ) \big)    \nonumber \\
&= \frac{1 - e^{i D M N ( n_{0} ) r n_{0}^{- t - k - 1}}}{1 - e^{i D M r n_{0}^{- t - k - 1}}} + \CO \big( N ( n_{0} )^{3} r n_{0}^{- t - k - 2} \big) .  \label{a25}
\end{align}
Since $D M N ( n_{0} ) r n_{0}^{- t - k - 1}$ is at distance less than $D M r n_{0}^{- t - k - 1}$ from $2 \pi \Z$, \eqref{a24} yields
\begin{equation*}
\bigg\vert \frac{1 - e^{i D M N ( n_{0} ) r n_{0}^{- t - k - 1}}}{1 - e^{i D M r n_{0}^{- t - k - 1}}} \bigg\vert \lesssim \frac{D M r n_{0}^{- t - k - 1}}{D M r n_{0}^{- t - k - 1}} = 1 .
\end{equation*}
On the other hand, \eqref{a22} and \eqref{a21} give
\begin{align*}
N ( n_{0} )^{3} r n_{0}^{- t - k - 2} &= N ( n_{0} ) \big( N ( n_{0} ) r n_{0}^{- t - k - 1} \big) \big( N ( n_{0} ) n_{0}^{- 1} \big)   \\
&\lesssim N ( n_{0} ) r^{- \delta ( t + k )} .
\end{align*}
Now the lemma follows from \eqref{a25} and the last two inequalities.
\end{proof}

Let $H$ be a fixed integer. Combining \eqref{a26}, \eqref{a21} and Lemma \ref{a27}, we get
\begin{align*}
\bigg\vert \frac{1}{N ( n_{0} ) - H} \sum_{n = 0}^{N ( n_{0} ) - H - 1} e^{i a_{n_{0} + n}^{( k )} ( h_{1} , \ldots , h_{k} )} \bigg\vert &\leq \frac{N ( n_{0} )}{N ( n_{0} ) - H} \bigg\vert \frac{1}{N ( n_{0} )} \sum_{n = 0}^{N ( n_{0} ) - 1} e^{i d_{n}} \bigg\vert + \frac{H}{N ( n_{0} ) - H}  \\
&\lesssim \bigg\vert \frac{1}{N ( n_{0} )} \sum_{n = 0}^{N ( n_{0} ) - 1} e^{i d_{n}} \bigg\vert + N ( n_{0} )^{- 1} \\
&\lesssim r^{- \delta ( t + k )} + r^{- \delta ( t + k + 1)}   \\
&\lesssim r^{- \delta ( t + k )} ,
\end{align*}
for $r$ large enough, where $r^{- \delta ( t + k )}$ tends to $0$ as $r \to \infty$. Then, the iterative Van der Corput argument below Lemma \ref{a17} with $J = k$ implies that, for $r$ large enough, we have
\begin{equation} \label{a28}
\bigg\vert \frac{1}{N ( n_{0} )} \sum_{n = n_{0}}^{n_{0} + N ( n_{0} ) - 1} e^{i a^{( 0 )}_{n}} \bigg\vert \leq \frac{\varepsilon t}{4 K} , \quad \forall n_{0} \in I_{k}.
\end{equation}

We now estimate the sum of $v_{n}$ defined in \eqref{a35}. For $n \in [ n_{0} , n_{0} + N ( n_{0} ) - 1]$, \eqref{a21} yields
\begin{equation} \label{a29}
\Big\vert \frac{1}{n} - \frac{1}{n_{0}} \Big\vert = \frac{n - n_{0}}{n n_{0}} \leq \frac{N ( n_{0} )}{n n_{0}} \leq C \frac{r^{- \delta ( t + k )}}{n} ,
\end{equation}
for some $C > 0$. Applying \eqref{a28} and two times \eqref{a29}, we deduce
\begin{align}
\bigg\vert \sum_{n = n_{0}}^{n_{0} + N ( n_{0} ) - 1} v_{n} \bigg\vert &= \bigg\vert \sum_{n = n_{0}}^{n_{0} + N ( n_{0} ) - 1} \frac{1}{n} e^{i a_{n}^{( 0 )}} \bigg\vert   \nonumber \\
&\leq \bigg\vert \sum_{n = n_{0}}^{n_{0} + N ( n_{0} ) - 1} \frac{1}{n_{0}} e^{i a_{n}^{( 0 )}} \bigg\vert + C r^{- \delta ( t + k )} \sum_{n = n_{0}}^{n_{0} + N ( n_{0} ) - 1} \frac{1}{n} \nonumber  \\
&\leq \frac{\varepsilon t}{4 K} \frac{N ( n_{0} )}{n_{0}} + C r^{- \delta ( t + k )} \sum_{n = n_{0}}^{n_{0} + N ( n_{0} ) - 1} \frac{1}{n} \nonumber \\
&= \frac{\varepsilon t}{4 K} \sum_{n = n_{0}}^{n_{0} + N ( n_{0} ) - 1} \frac{1}{n_{0}} + C r^{- \delta ( t + k )} \sum_{n = n_{0}}^{n_{0} + N ( n_{0} ) - 1} \frac{1}{n} \nonumber \\
&\leq \frac{\varepsilon t}{4 K} \sum_{n = n_{0}}^{n_{0} + N ( n_{0} ) - 1} \frac{1}{n} + 2 C r^{- \delta ( t + k )} \sum_{n = n_{0}}^{n_{0} + N ( n_{0} ) - 1} \frac{1}{n} ,
\end{align}
for all $n_{0} \in I_{k}$. Thus, we just proved

\begin{lem}\sl \label{a30}
For $r$ large enough, we have
\begin{equation*}
\bigg\vert \sum_{n = n_{0}}^{n_{0} + N ( n_{0} ) - 1} v_{n} \bigg\vert \leq \frac{\varepsilon t}{2 K} \sum_{n = n_{0}}^{n_{0} + N ( n_{0} ) - 1} \frac{1}{n} ,
\end{equation*}
uniformly for $n_{0} \in I_{k}$.
\end{lem}

For $r$ large enough, we decompose $I_{k}$ in a disjoint union of intervals
\begin{equation*}
I_{k} = \bigg( \bigcup_{\ell = 1}^{L} \CI_{\ell} \bigg) \cup \CI_{\infty} ,
\end{equation*}
such that $\CI_{\ell} \cap \N$ is of the form $\{ n_{\ell} , \ldots , n_{\ell} + N ( n_{\ell} ) \} \subset I_{k}$ for all $\ell = 1 , \ldots , L$ and $\CI_{\infty}$ is too small to contain such a period. In particular, \eqref{a21} implies
\begin{equation} \label{a31}
\Big\vert \sum_{n \in \CI_{\infty}} v_{n} \Big\vert \lesssim r^{- \delta ( t + k )} \leq 1 ,
\end{equation}
for $r$ large enough. Using Lemma \ref{a30} to estimate the sum over $\CI_{\ell}$ for $\ell = 1 , \ldots , L$ and \eqref{a31} to estimate the sum over $\CI_{\infty}$, we get
\begin{equation} \label{a32}
\Big\vert\sum_{n \in I_{k}} v_{n} \Big\vert \leq \frac{\varepsilon t}{2 K} \sum_{n \in I_{k}} \frac{1}{n} + 1 \leq \frac{\varepsilon t}{2 K} \sum_{n \leq r^{\frac{1}{t}}} \frac{1}{n} \leq \frac{\varepsilon}{K} \ln r ,
\end{equation}
for $r$ large enough.

\underline{Step 8:} conclusion. Combining \eqref{a10} for the large values of $n$, \eqref{a11} for the main contribution, \eqref{a34} for the small values of $n$, \eqref{a33} for the contribution of $J_{k}$ and \eqref{a32} for the contribution of $I_{k}$, we obtain
\begin{align*}
\Big\vert g_{t} ( r ) - \frac{\ln r}{2 t} \Big\vert &\leq \bigg\vert \sum_{n < r^{\frac{1}{t} + \delta}} u_{n} - \frac{\ln r}{2 t} \bigg\vert + \varepsilon \ln r   \\
&\leq \bigg\vert \sum_{n < r^{\frac{1}{t} + \delta}} \frac{1}{2 n} - \frac{\ln r}{2 t} \bigg\vert + \bigg\vert \sum_{n < r^{\frac{1}{t} + \delta}} v_{n} \bigg\vert + \varepsilon \ln r  \\
&\leq \bigg\vert \sum_{n < r^{\frac{1}{t + K} - \delta}} v_{n} \bigg\vert + \sum_{k = 0}^{K - 1} \bigg\vert \sum_{n \in I_{k}} v_{n} \bigg\vert + \sum_{k = 0}^{K} \bigg\vert \sum_{n \in J_{k}} v_{n} \bigg\vert + 2 \varepsilon \ln r  \\
&\leq \varepsilon \ln r + \sum_{k = 0}^{K - 1} \frac{\varepsilon}{K} \ln r + \sum_{k = 0}^{K} \frac{\varepsilon}{K + 1} \ln r + 2 \varepsilon \ln r  \\
&\leq 5 \varepsilon \ln r ,
\end{align*}
for $r$ large enough. This ends the proof of Proposition \ref{pr2}.

\subsection*{Acknowledgements}
An essential part of this   work  was done during the visit of the third author to the University of Bordeaux, France, in January 2017, and to the Institute of Mathematics, Bulgarian Academy of Sciences, in December 2017. He thanks these institutions for hospitality and financial support. He also gratefully acknowledges the partial support of the Chilean Scientific Foundation {\em Fondecyt} under Grant 1170816. The authors thank Grigori Rozenblum for a discussion on the spectral
gaps adjoining the origin.

\end{document}